\documentclass[11pt,twoside]{article}
\usepackage[english]{babel}
\usepackage{amssymb}
\usepackage[mathscr]{eucal}
\usepackage{amsmath,amsthm}
\usepackage{mathrsfs}
\usepackage{lgreek}
\usepackage[shortlabels]{enumitem}
\usepackage[colorlinks=true,
   urlcolor=blue,           filecolor=green,      
   citecolor=green,      
   linkcolor=red,           bookmarks=true,
  unicode,
   plainpages=false,   ]{hyperref}
\usepackage{color}

\definecolor{royalblue}{rgb}{0,0,0.128}
\def\bp{\begin{proof}}
\def\ep{\end{proof}}
\def\n{\nabla}
\def\ssfrac#1#2{\mbox{\large$\frac{#1}{#2}$}}
\def\sfrac#1#2{\mbox{\Large$\frac{#1}{#2}$}}
\def\intl#1{\int\limits_{#1}}
\def\intll#1#2{\int\limits_{#1}^{#2}}
\def\dm {|\hskip-0.05cm|}
\def\tm{|\hskip-0.05cm|\hskip-0.05cm|}
\def\OO{\Omega}
\def\displ{\displaystyle}
\def\VSE{\vspace{6pt}\\&\displ }
\def\VS{\vspace{6pt}\\\displ }
\def\VSs{\vspace{4pt}\\\displ}
\def\rf#1{{\rm(\ref{#1})}}

\newcommand{\R}{\mathbb{R}}
\newcommand{\N}{\mathbb{N}}
\def\à{à}

\def\dy{\displaystyle}
\def\vep{\varepsilon}

\def\be{\begin{equation}}
\def\ba{\begin{array}}
\def\ea{\end{array}}
\def\ee{\end{equation}}
\def\vs1{\vspace{1ex}}

\def\ov{\overline}

\def\Ã©{\'{e}}
\def\Ãš{\`{e}}
\newtheorem{lemma}
{\bf Lemma} 
\font\sc=cmcsc10
\title{\Large\bf The Navier-Stokes Cauchy problem\\in a class of weighted function  spaces}
\author{\sc  Paolo Maremonti and Vittorio Pane
\thanks{Dipartimento di Matematica e Fisica,  
Universit\`{a} degli 
Studi della Campania
``L. Vanvitelli'', via Vivaldi 43, 81100 \null\hskip0.55cmCaserta,
 Italy.\newline\null\hskip0.55cm
paolo.maremonti@unicampania.it \newline \null\hskip0.55cm vittorio.pane2@unicampania.it\newline\null\hskip0.55cm The  research activity  is performed under the
auspices of   GNFM-INdAM. }}
\date{\today}
\begin{document}
\markboth{\footnotesize\rm    P. Maremonti and V. Pane} {\footnotesize\rm
The Navier-Stokes Cauchy problem in a class of weighted function  spaces}
\maketitle 
{\small\bf Abstract} - {\small We consider the Navier-Stokes Cauchy problem with an initial datum in a weighted Lebesgue space. The weight is a  radial function increasing at infinity.  Our study partially follows the ideas of the paper by G.P.Galdi \& P.Maremonti {\it ``On the stability of steady-state solutions to the Navier-Stokes equations in the whole space''}, JMFM, {\bf 25} (2023). The authors of the quoted paper consider a special study of  stability of steady fluid motions. The results hold in 3D and for small data. Here, relatively to the perturbations of  the rest state, we generalize the result. We study the $n$D Navier-Stokes Cauchy problem, $n\geq3$. We prove the existence (local) of a unique regular solution. Moreover, the solution enjoys a spatial asymptotic decay whose order of decay is connected to  the weight.}
\noindent
\newcommand{\red}{\protect\bf}
\renewcommand\refname{\centerline
{\red {\normalsize \bf References}}}
\newtheorem{ass}
{\bf Assumption} 
\newtheorem{defi}
{\bf Definition} 
\newtheorem{theorem}
{\bf Theorem} 
\newtheorem{rem}
{\sc Remark} 
\newtheorem{coro}
{\bf Corollary} 
\newtheorem{prop}
{\bf Proposition} 
\renewcommand{\theequation}{\arabic{equation}}
\setcounter{section}{0}
\section{\label{SC-I}Introduction}
In the recent paper \cite{GM} the authors consider  a problem of stability for steady solutions to the Navier-Stokes equations in the whole space. Actually, in the paper are realized  results both for steady   and for unsteady solutions to the  Navier-Stokes equations. But related to the unsteady problem  two chief different questions are achieved. The former is  a result of  global existence for small data, obtained via   a new  scaling invariant space, that is  the weighted function space $L^p_\alpha(\R^3)$, $p>3$ and $\alpha=1-\frac3p$ ($|x|^\alpha$ weight, see \rf{DSP}). Denoted by $u(t, x)$ the solution, they achieve  $u\in L^p_\alpha(\R^3)$ for all $t>0$ with a uniform bound $c(u_0)$, and $t^\frac12\dm  u(t)\dm _\infty\leq c(u_0)$ for all $t>0$. Being the weighted space embedded  in a suitable scaling invariant  Besov's space, this result  seems to be   not new. But it allows to prove the latter result.   That  is to deduce space-time asymptotic behavior as negative power of $t$ and $|x|$, respectively.  Respect to both the variables the order of decay (see below) is imposed by the exponent $\alpha$ of the weight $|x|$. Since the metric is scaling invariant  and $\alpha<1$, then the scaling value  $-1$ of the solution is compensated  in multiplicative form  by the time variable with exponent $(1-\alpha)/2$\,.\par This note works on the wake of the paper \cite{GM}. However, we consider the $n$D case, perturbations to the rest state and large data. Consequently,  we just achieve a theorem  of local existence in time. Nevertheless, we state the spatial asymptotic behavior of the solutions.
\par To better state and discuss our chief results we introduce some notation and the statements of the chief theorems.
 \par For given $p>n$ and $\alpha=1-\frac{n}{p}$, by $L^p_{\alpha}(\R^n)$ we denote the following weighted Lebesgue space
\begin{equation}\label{DSP}
\{u\, :\, u|x|^{\alpha} \in L^p(\R^n)\}
\end{equation}
endowed with the natural norm
\begin{equation}
\dm   u\dm _{\alpha,p}:=\dm u|x|^{\alpha}\dm _p\,.
\end{equation}
In the following $p'$ is the coniugate exponent of exponent  $p$ and we set $\alpha':=-\alpha$. By the symbol  $L^{p'}_{\alpha'}(\R^n)$ we mean th dual space of $L^p_\alpha(\R^n)$. Moreover, we set
\[\mathscr{C}_0(\R^n):=\{\phi \in \mathit{C}_0^{\infty}(\R^n)\, :\, \n \cdot \phi=0\}\,,\]
and by $J^{p}_{\alpha}(\R^n)$, we denote the completion of $\mathscr{C}_0(\R^n)$ in $L^p_{\alpha}(\R^n)$-norm\footnote{For more details, see Theorem\,\ref{HWD} and appendix,  and  \cite{Fr-UF,HDS} too.}.\\
We consider the problem:
\begin{equation}\label{problem}
\ba{ll}
u_t + u \cdot \nabla u + \nabla \pi= \Delta u \,,\; &\text{in}\,\, (0,T) \times \R^n\,, \VSs
\nabla \cdot u=0\,, &\text{in}\,\, (0,T) \times \R^n \,,\VSs
u(0,x)=u_0(x)\,, \; &\text{on}\,\, \{0\} \times \R^n\,,
\ea
\end{equation}
under the assumption that $u_0 \in L^p_{\alpha}(\R^n)$ and $\n\cdot u_0=0$ in weak sense.
\par Let $a\in L^p_\alpha(\R^n)$. For $t>0$ and $\rho>0$,  we set 
\begin{equation}\label{K(t,ro)}
K_a(t,\rho):=c_0 \dm a\dm _{\alpha,p,\rho} + c_1 e^{-\frac{\rho ^2}{8t}}\dm a\dm _{\alpha,p}\,,
\end{equation}
where
\[\dm a\dm _{\alpha,p,\rho}:=\sup_{x \in \R^n}\Big[\int_{B_{\rho}(x)}|a(y)|^p|y|^{\alpha p}\dy\Big]^\frac1p\,,\]
and, for $q>p$, we set
\be\label{TM}\tm u(t)\tm:=t^{\frac n{2p}}\dm u(t)|x|^{\alpha}\dm_\infty+t^\frac12\dm u(t)\dm_\infty+ t^{\frac n2\left(\frac1n-\frac1q\right)}\dm u(t)\dm_q\,.\ee
In the following, with the symbol $X(p)$ we denote the following spaces:
$$X(p):=\left\{\hskip-0.5cm\ba{lll}&L^{p'}_{\alpha'}(\R^n)\cap L^{q'}(\R^n)\,,q\in(p,2n)\,,&\mbox{for } p\in(n,2n)\,,\VSE L^{p'}_{\alpha'}(\R^n)\cap L(\ssfrac n{n-1},1)(\R^n)\,,&\mbox{for }p=2n\,,\VSE L^{p'}_{\alpha'}(\R^n)\cap L(\ssfrac{np}{p(n-2)+2n},1)(\R^n)\,,&\mbox{for }p>2n\,.\ea\right.$$
We are interested to prove the following
\begin{theorem}\label{esistenza}{\sl 
Let $u_0 \in L^p_{\alpha}(\R^n)$ and divergence free in weak sense. There exists an istant $T:=T(u_0)$ such that problem \eqref{problem} admits a solution $(u,\pi)$, with $u$ given by \rf{RI}, defined and smooth for all $t \in (0,T)$, such that 
  for all $q>p$, 
\begin{equation}\label{andamento_sol_q}
t^\frac n{2p}\dm u(t)|x|^\alpha\dm_\infty+t^{\frac{1}{2}}\dm u(t)\dm _{\infty}+t^{\frac{n}{2}(\frac{1}{n}-\frac{1}{q})}\dm u(t)\dm _q \leq K(t,\rho)+  c\dm u_0\dm _{\alpha,p}^{1-\gamma}K(t,\rho)^{\gamma}\,,\mbox{ for all }t\in(0,T)\,,
\end{equation}
where $\gamma=1-\frac{n}{q}$. 
The following limit properties hold:
\begin{equation}\label{limite_linfinito_soluzione}
\lim_{t \to 0^+} \tm u(t)\tm=0\,, 
\end{equation} \be \label{weak_convergence1}
\lim_{t\to 0}\,(u(t),\varphi)=(u_0,\varphi) \,, \text{ for all } \varphi \in  X(p)\,.
\ee
Moreover, we get\begin{equation}\label{andamento_pressione}
t^{1-\frac{n}{q}}\dm \pi\dm _{\frac{q}{2}}\leq \big[K(t,\rho)+  c\dm u_0\dm _{\alpha,p}^{1-\gamma}K(t,\rho)^{\gamma}\big]^2\,, \quad \text{for all}\,\, t \in (0,T)\,,
\end{equation}
where $c$ is a constant indipendent of $u$.
Finally, if the norm $\dm u_0\dm _{\alpha,p}$ is suitably small, then above results hold for all $t>0$.}
\end{theorem}
\begin{rem}\label{R-I}{\rm As it is well known, by means of standard arguments, see {\it e.g.} \cite{Knightly}, based on the representation formulas \rf{RI}-\rf{RIT},   for a solution $(u,\pi)$ furnished in Theorem\,\ref{esistenza}   one proves that for the solution $(u,\pi)$ the pointwise regularity holds  for all $(t, x)\in (\eta,T)\times\R^n$, for all $\eta>0$.
}\end{rem}
\begin{theorem}[Uniqueness]\label{unicita}{\sl
Let $(u,\pi)$ be a solution to Cauchy problem \eqref{problem} whose existence is ensured by Theorem\,\ref{esistenza}. Then $(u,\pi)$, up to a function $h(t)$ for the pressure field, is unique in the class of existence detected in Theorem\,\ref{esistenza}.}
\end{theorem}     
We give a possible reading of the behavior for large $|x|$. \par Roughly speaking,  we can state that usually to a weak initial datum $u_0$, by means of the equations, corresponds for $t>0$ a ``regularizing'' effect for the field $u_0$ defined on $\{0\}\times\R^n$, that is $u(t, x)$ is  regular   in $(0,T)\times\R^n$. Here, we find fine the fact that to a weak asymptotic property, that is $u_0\in L^p_\alpha$,  allow us to achieve an  asymptotic spatial behavior.    We believe that this is interesting because it indicates that the equations and the weighted  integrability  realize a ``spatial stabilizing effect'' too. \par This is different from the usual properties of spatial decay as given, remaining closely connected to the case of perturbations to the rest state, in \cite{CM-AB,Knightly}. In these papers the profile of the spatial asymptotic behavior of $u_0$ is given on $\{0\}\times\R^3$. \par In a next paper we investigate about these properties in the case of global weak solutions\footnote{\, In order to obtain the existence of a weak solution with data $u_0\in J^p_\alpha(\R^3)$ (scaling invariant space) we employ the approach given in \cite{M-WS} for the same question but with an initial datum $u_0\in J^3(\OO)$ (scaling invariant space).}.  In the light to prove, {\it mutatis mutandis}, the equivalent of the results obtained in \cite{CKN} (considered in \cite{CM-DC} under slightly weaker hypotheses),  where, provided that the initial datum $u_0$ is in $H^1(\R^3)$, the corresponding solution is regular in the exterior of a ball. Heuristically, some smallness of datum for large $|x|$ is enough to ensure regularity and, in our setting not only, also a suitable asymptotic behavior for all $(t, x)\in (0,\infty)\times \R^3- B_R$\,.
\par We would like to give the outlines of the proof and comments to the results. We set
\be\label{RI}u(t,x)= u^0(t,x)+\ov{u}(t,x)\,, \ee
where
\begin{equation}
u^0(t,x):=H\ast u_0(t,x)
\end{equation}
and
\be\label{RIT}\ov u(t,x):= - \int_0^t \int_{\R^n} \nabla E(t-\tau,x-y)u(\tau,y)\cdot u(\tau,y)\,dy\,d\tau\,.\ee
For the existence we follow the approach recently given in the paper \footnote{\label{EMCM}\,We take the opportunity to point out that in the print of the paper \cite{CM-L3} the ratio line in the exponents has been removed. This has been probably made by the editor  for a style question that provides for  the slash symbol in the ratio, but not suitably substituted in the formulas.  As a consequence   to a nonspecialist of semigroup properties, related to solutions of parabolic equations, some exponents  appear  with no meaning.}  
\cite{CM-L3}. More precisely, we employ the Oseen representation formula (see \rf{RI}-\rf{RIT}) that we solve via the method of the successive approximations. Denoted by $\{u^m\}$ the sequence, using the scaling invariant functional \rf{TM}, we achieve the convergence   by means of  a special decomposition  of the linear part (in this connection see  \cite{CM-KC} too).  The decomposition  furnishes smallness properties of the $\tm u^m(t)\tm$ in a neighborhood  of $t=0$,   that ensure the convergence of the sequence on some interval $(0,T(u_0))$, $u_0$ initial datum.  By stating the existence of a solution $u$ having finite norm  $\tm\cdot\tm$, hence,  {\it a priori}, not finite weighted $L^p_\alpha$-norm  of the initial datum (see \rf{andamento_sol_q}),  we get the spatial asymptotic behavior of the solution and its time singular  behavior in a right neighborhood of $t=0$, both dependending on $\alpha=1-\frac np$. 
\par We believe that  the uniqueness result is some sense intriguing. Actually, we go to state a uniqueness result just considering the special metric $\tm\cdot\tm$ and for a solution that achieves the initial datum in the weak form \rf{weak_convergence1}. Hence, our solution does not belong to a functional space with finite norm uniformly in $t\in(0,T)$, in particular that of the initial datum (as usual, in particular as in the case of the paper \cite{GM} closely connected with this), and without  requiring smallness to the size of some suitable norm of the solutions. We bridge the quoted gaps  with the following strategy. Firstly, we recall  that any solution $u$, enjoying the same regularity of our solution, can be written as the sum of a solution to the Stokes problem, say $u^0(t, x)$ that brought  the initial datum $u_0(x)$, and another solution $\ov u(t, x)$ that is related to a Stokes problem that has the convective term $u\cdot \n u$ as ``body force'' and has homogeneous initial datum (see formulas \rf{RI}-\rf{RIT}). Then we can claim that if $u\in L^{\ov q}(\R^n)$,  for a suitable $\ov q>p$, then $\ov u\in L^{\frac {\ov q}2}(\R^n)$ Lebesgue's space or, by virtue of the spatial decay, some suitable Lorentz's space (see Lemma\,\ref{andamento_a_0_parte_non_lineare} for the details). Moreover,  the property $\displ\lim_{t\to0}\dm  \ov u(t)\dm _\frac{\ov q}2=0$  or those related to the possible Lorentz norm hold.   Since the linear parts are trivially  coincident  for any pair of solutions   both corresponding to $u_0$ (see Theorem\,\ref{SPWS}), that is $u=u^0+\ov u$ and $v=u^0+\ov v$, then the difference $w:=u-v$ can be write as $w=\ov u-\ov v$.  The difference $w$ can be seen as solution to a special integral equation (see \rf{ultima}) that we discuss by means of an argument of duality as introduced by Foias in \cite{F} for the Navier-Stokes equations. The duality argument leads to apply a special Gronwall's lemma (see Lemma\,\ref{GRLNL}) that works with Gronwall's disequations in $h$ (see \rf{GRLNI})  defined by means of kernel weakly singular,  and  function $h$ defined a.e. in $t>t_0$. This strategy allow us to achieve the uniqueness in our class of existence.\par The plan of the paper is the following. In sect. 2 we furnish some preliminaries of notations and results related to some interpolation inequalities, to a generalized Gronwall's  lemma and to the solutions to the Stokes problem.  In sect. 3 we furnish the proof of Theorem\,\ref{esistenza} and Theorem\,\ref{unicita}.
\section{Preliminary results}
In this section, we recall some notions that will be employed in our proofs and we rewrite \eqref{problem} as an integral equation. Succesively, we prove some preliminary results. Hereafter we use the symbol $D_x^h$ to denote the partial derivatives $D_x^{\alpha}$
with respect the variable $x \in \R^n$, where $\alpha$ is a multi-index with $|\alpha|=h$. Also, we use the symbol $D_t^k$ to denote  the $k$-th partial derivative with respect the real variable $t$.\\
Firstly, we introduce the Lorentz spaces\footnote{\,For a wider background, we refer
the reader, {\it e.g.}, to  \cite{PKJF-LS}.}. Let $\Omega \subseteq \R^n$ and $g(x)$ be a (Lebesgue) measurable function, we denote by $\mu_g$ the distribution function of $g$, that is:
\[\mu_g(\lambda):=meas(\{x\in \Omega\,:\, |g(x)|>\lambda\})\quad \text{for all}\,\, \lambda>0,\]
and we define the nonincreasing rearrangement of $g$ as:
\[g^*(t):=\inf\{\lambda\,:\, \mu_g(\lambda)\leq t\} \quad \text{for}\,\, t\in[0,\infty).\]
Now, assume that $0<p\leq q \leq \infty$. The Lorentz space $L(p,q)(\Omega)$ is the collection of all measurable function $g(x)$ such that $\dm g  \dm _{(p,q)}<\infty$, where:
\begin{equation}
\dm g  \dm _{(p,q)}=
\begin{cases}
\bigg(\mbox{\Large$ \int$}_0^{\infty}[t^{\frac{1}{p}}g^*(t)]^q\,  \frac{dt}{t}\bigg)^{\frac{1}{q}}\quad &\text{if}\,\, 0<q<\infty,\\
\displaystyle\sup_{t\in (0,\infty)} t^{\frac{1}{p}}g^*(t) \quad &\text{if}\,\, q=\infty.
\end{cases}
\end{equation}
\\
Additionally, we consider the Helmoltz decomposition for the space $L^q_{\beta}$, $q\in(1,\infty)$. Precisely, setting:
\begin{gather*}
J^q_{\beta}(\mathbb{R}^n)=\{v\in L^q_{\beta}(\mathbb{R}^n)\, |\, (v,\nabla h)=0 \,\, \text{for all}\, h \in W^{1,q'}_{loc}(\mathbb{R}^n,\beta ')\text{,}\, \nabla h \in L^{q'}_{\beta'}(\R^n)\}\\
G^q_{\beta}(\mathbb{R}^n)=\{u\in L^q_{\beta}(\mathbb{R}^n)\, |\, \exists h\, : \,\, u=\nabla \pi \text{,}\,\text{with} \, \pi \in W^{1,q}_{loc}(\mathbb{R}^n,\beta ')\text{,}\, \nabla \pi \in L^{q}_{\beta}(\R^n)\} 
\end{gather*}
we have the following result:
\begin{theorem}\label{HWD}{\sl
Let $q \in (1,+\infty)$ and $\beta \in (-\frac{n}{q},\frac{n}{q'})$. Then,
\begin{equation}
L^q_{\beta}(\R^n)=J^q_{\beta}(\mathbb{R}^n)\oplus G^q_{\beta}(\mathbb{R}^n),
\end{equation}
holds, that is for all $u \in L^q_{\beta}(\R^n)$, $u=v+\nabla \pi_u$, with the following integral identities:
\begin{gather}
(u,\nabla \pi)=(\nabla \pi_u ,\nabla \pi) \quad \text{for all}\, \, \nabla \pi \in G_{\beta '}^{q'}(\R^n), \\
(v,\nabla \pi)=0 \quad \text{for all}\, \, \nabla \pi \in G_{\beta '}^{q'}(\R^n), 
\end{gather}
and
\begin{equation} 
\dm v\dm _{q,\beta} + \dm  \nabla \pi \dm _{q,\beta} \leq C \dm u\dm _{q,\beta},
\end{equation}
whith $C$ indipendent of $u$. Finally, $J^q_\beta(\R^n)$ coincides with the completion of $\mathscr C_0(\R^3)$ in $L^q_\beta(\R^3)$-norm.}
\end{theorem}
\begin{proof}
For the sake of the completeness we give the proof in the  Appendix, Theorem\,\ref{dec_helmotz_teorema}. However, we also quote, {\it e.g.}, \cite{Fr-UF,HDS}. 
\end{proof}
In the following, we denote by $H(t,x)$ the fundamental solution of the heat equation, that is
\begin{equation*}
H(t,x):=\frac{1}{(4\pi t)^{\frac{n}{2}}}e^{-\frac{|x|^2}{4t}}\,,
\end{equation*}
and with $E(s,x)$ the Oseen tensor, fundamental solution of the Stokes system, with components
\begin{gather*}
E_{ij}(t,x):=-H(t,x)\delta_{ij}+ D_{x_i x_j}\phi(t,x)\\
\phi(t,x):= \mathcal{E}(x)\sfrac{1}{(2\sqrt t)^{n-2}}\int_0^{|x|}\exp\Big[-\ssfrac{a^2}{4t}\Big]\,da
\end{gather*}
where $\mathcal{E}$ is the fundamental solution of the Laplace equation.     Let $h+k\geq 1$, then the following estimates hold\footnote{ \, For more details about the Oseen tensor's properties see \cite{Knightly} or \cite{Os}.}:
\begin{gather}
|D^k_t D^h_x H(t,x)| \leq c t^{-\frac{n}{2}-k-\frac{h}{2}}e^{-\frac{|x|^2}{4t}}, \quad \text{for all}\,\, t>0 \,\, \text{and}\,\, x\in \R^{n}  \label{EHE}\,,\\
|D_{t}^k D_{x}^h E(t,x)| \leq c(|x|+t^{\frac{1}{2}})^{-n-h-k}, \quad \text{for all}\,\, t>0 \,\, \text{and}\,\, x\in \R^{n} \,. \label{prop_tens_oseen_derivate}
\end{gather}
We look for a solution to the problem \eqref{problem} by means of nonlinear Oseen integral equation
\begin{equation*}
u(t,x)= H\ast u_0 (t,x) + \int_0^t \int_{\R^n} E(t-\tau,x-y)u(\tau,y)\cdot \nabla u(\tau,y)\,dy\,d\tau\,.
\end{equation*}
After an integration by parts on $\R^n$, the last equation is equivalent to
\begin{equation}\label{equazione_integrale}
u(t,x)= H\ast u_0 (t,x) - \int_0^t \int_{\R^n} \nabla E(t-\tau,x-y)u\otimes u(\tau,y)\,dy\,d\tau\,.
\end{equation} 
\subsection{Some interpolation inequalities}
We provide two results concerning interpolation type inequalities. We have the following:
\begin{lemma}\label{IWSLq}{\sl Let $p>n$ and $r\in[p,\infty]$. 
Let $g\in  L^{p}_{\alpha}(\R^n)\cap L^{r}(\R^n)$ with $\alpha:=1-\frac{n}{p}$. Then, for all $q\in (n,r)$, there exists a positive constant $c$ indipendent of $g$ such that
\begin{equation}\label{interpolazione}
\dm g\dm_q\leq c \dm g \dm_{\alpha,p}^{1-\theta}\dm g\dm_r^{\theta}\, , \quad \mbox{with } \theta=\sfrac{r(q-n)}{q(r-n)}\, .
\end{equation}}
\end{lemma}
\begin{proof}
We start proving the relation \eqref{interpolazione} for all $q \in (n,p)$. \par
Let us consider $R>R_0$, with $R_0$ to be fixed. We set $C_{R_0,R}:=B_R\setminus B_{R_0}$. In our hypotheses we can assume $g\in L^q(B_R)$. Hence, we get:
\be \label{interp1}
\dm g \dm_{L^q(B_R)}\leq  \dm g \dm_{L^q(B_{R_0})}+\dm g \dm_{L^q(C_{R_0,R})}. 
\ee
In order to estimate the first term on the right side-hand, using H\"older's inequality, we calculate
\be \label{SI:L1}
\dm g \dm_{L^1({B_{R_0}})}= \int_{B_{R_0}}|g(x)| |x|^{\alpha}|x|^{-\alpha}\,dx\leq \dm g\dm_{\alpha,p}\bigg(\int_{B_{R_0}}|x|^{-\alpha p'}\bigg)^{\sfrac{1}{p'}}= c \dm g\dm_{\alpha,p} R_0 ^{n-1}.
\ee
By  interpolation inequality for Lebesgue's spaces, taking in account last relation, we have:
\be\label{SI:Bn}\dm g \dm_{L^q({B_{R_0}})}\leq \dm g \dm_{L^1({B_{R_0}})}^{1-a}\dm g \dm_{L^r({B_{R_0}})}^{a}\leq c \dm g \dm_{\alpha,p}^{1-a}\dm g \dm_{L^r({B_{R_0}})}^a R_0^{(n-1)(1-a)}\, ,\ee
with $a=\sfrac{r(q-1)}{q(r-1)}$. Now, we consider $\dm g \dm_{L^q(C_{R_0,R})}$, the second term in the right-hand side of \eqref{interp1}. Since $q<p$, applying H\"older's inequality with exponent $p$, $q'$ and $\ssfrac{pq}{p-q}$, we get
\begin{equation*}
\begin{split}
\dm g\dm_{L^q(C_{R_0,R})}^q&=\int_{C_{R_0,R}}|g||x|^{\alpha}|g|^{q-1}|x|^{-\alpha}\,dx \leq \dm g\dm _{\alpha,p}\dm g\dm_{L^q(C_{R_0,R})}^{q-1}\bigg(\int_{C_{R_0,R}}|x|^{-\alpha \ssfrac{pq}{p-q}}\bigg)^{\ssfrac{p-q}{pq}}\\
&=c\, \dm g\dm _{\alpha,p}\dm g\dm_{L^q(C_{R_0,R})}^{q-1}\, R_0^{\, \ssfrac{n(p-q)}{pq}-\alpha}, 
\end{split}
\end{equation*}
from which, recalling the expression of $\alpha$ and dividing by $\dm g\dm_{L^q(C_{R_0,R})}^{q-1}$, we arrive at
\begin{equation}\label{int3}
\dm g\dm_{L^q(C_{R_0,R})}\leq c\, \dm g \dm _{\alpha,p}\, R_0^{\, -1+\frac{n}{q}}\, .
\end{equation}
Then, summing \eqref{SI:Bn} and \eqref{int3} we have:
\begin{equation*}
\dm g\dm_{L^q(B_R)}\leq c \dm g \dm_{\alpha,p}^{(1-a)}\, \dm g\dm_r^{a}\, R_0^{\,(n-1)(1-a)} + \dm g \dm _{\alpha,p}\, R_0^{\, -1+\frac{n}{q}}\, . 
\end{equation*}
Setting $R_0:= \Big(\frac{\dm g\dm_{\alpha,p}}{\dm g\dm_r}\Big)^{\frac{1}{1-\frac nr}}$, we arrive at 
\begin{equation}
\dm g\dm_{L^q(B_R)}\leq c \dm g \dm_{\alpha,p}^{\ssfrac{n(r-q)}{q(r-n)}}\dm g\dm_r^{\ssfrac{r(q-n)}{q(r-n)}}\,.
\end{equation}
Since the right-hand side is independent of $R$, letting $R\to\infty$, we get  \rf{interpolazione} for $q \in (n,p)$\,.\par 
Let us consider $q\in [p,r)$. Employing the same arguments used to obtain estimate \rf{SI:Bn}, we get:
\begin{equation}\label{SI:Bp}
\begin{split}
\dm g \dm_{L^q(B_{R_0})}\leq \dm g\dm_{L^1(B_{R_0})}^{1-a}\dm g \dm_{L^r(B_{R_0})}^{a}\leq \dm g \dm_{\alpha,p}^{1-a}\dm g \dm_{L^r(B_{R_0})}^{a}R_0^{(n-1)(1-a)}\,,
\end{split}
\end{equation}
with $a=\ssfrac{r(q-1)}{q(r-1)}$.\par
For the term $\dm g \dm_{L^p(C_{R,R_0})}$ we obtain
\begin{align*}
\dm g \dm_{L^p(C_{R,R_0})}^p= \int_{C_{R,R_0}}|g||x|^{\alpha}|g|^{p-1}|x|^{-\alpha}\,dx\leq R_0^{-\alpha}\dm g\dm_{\alpha,p}\dm g \dm_{L^p(C_{R,R_0})}^{p-1}
\end{align*}
from which we get
$$
\dm g \dm_{L^p(C_{R,R_0})}\leq R_0^{-\alpha}\dm g\dm_{\alpha,p}\,.
$$
By interpolation inequality on $(p,r)$ we also arrive at
\begin{equation}\label{SI:Cp}
\dm g \dm_{L^q(C_{R,R_0})}\leq \dm g \dm_{L^p(B_R)}^{1-d}\dm g \dm_{L^r(B_R)}^d\leq \dm g\dm_{\alpha,p}^{1-d}\dm g\dm_r^dR_0^{-\alpha(1-d)}
\ee
with $d=\sfrac{r(q-p)}{q(r-p)}$ and $1-d=\sfrac{p(r-q)}{q(r-p)}$, for all $q\in[p,r)$.
Summing \eqref{SI:Bp} and \eqref{SI:Cp} and setting $R_0:=\Big(\frac{\dm g\dm_{\alpha,p}}{\dm g\dm_r}\Big)^{\frac{1}{1-\frac nr}}$,  we get
\be\label{SI:BRp}\ba{ll}
\dm g\dm_{L^q(B_R)}&\leq R_0^{(n-1)(1-a)}\dm  g \dm_{\alpha, p}^{1-a}\dm g \dm_r^a +\dm g\dm_{\alpha,p}^{1-d}\dm g\dm_r^d R_0^{-\alpha (1-d)}\dm g\dm_{\alpha,p} 
\VSE=2\dm g\dm_{\alpha,p}^{\frac{n(r-q)}{q(r-n)}}\,\dm g \dm_r^{\frac{r(q-n)}{q(r-n)}}\,.\ea
\ee
Hence, letting $R \to \infty$, we obtain \eqref{interpolazione} for $q\in[p,\infty)$. \end{proof}
\begin{lemma}\label{LQFLPW}{\sl Assume $\OO\subset\R^n$  bounded and $g\in L(p,\infty)(\OO)$ Lorentz's space. Then for all $q\in[1,p)$ there exists $c(p,q)$ independent of $g$ such that
\be\label{LQFLPW-I}\dm  g\dm _q\leq c|\OO|^\alpha\dm  g\dm _{(p, \infty)}\,,\mbox{ with }\alpha:=\sfrac1q-\sfrac1p\,.\ee}\end{lemma} \bp
  Assuming $\OO$ bounded, it is well known that $L(p,\infty)(\OO)\subset L^q(\OO)$, for all $q\in [1,p)$. Hence we limit ourselves to prove the estimate \rf{LQFLPW-I}. Actually, for all $\rho>0$, we get $$\ba{ll}\dm  g\dm _q^q\hskip-0.2cm&\displ=q\intll0\infty\sigma^{q-1}m(\sigma,g)d\sigma=q\intll0\rho \sigma^{q-1}m(\sigma,g)d\sigma+q\intll\rho\infty\sigma^{q-1}m(\sigma,g)d\sigma\\&\displ\leq q|\OO|\intll0\rho\sigma^{q-1}d\sigma+q\dm  g\dm _{(p,\infty)}^p\intll\rho\infty \sigma^{q-1-p}d\sigma =|\OO|\rho^q+\sfrac q{p-q}\dm  g\dm _{(p,\infty)}^p\rho^{q-p}\,.\ea$$
Setting $\rho:=\Big[\frac{ q\dm  g\dm _{(p,\infty)}^p}{(p-q)|\OO|}\Big]^\frac1p$\,, we get \rf{LQFLPW-I}.
\ep  
\subsection{A generalized Gronwall's inequality}
Let us consider the following special Gronwall's inequality:
\be\label{GRLNI}\mu\in [0,1)\,,\; h(t)\leq \psi(t)+ \intll0t\sfrac{g(\tau,h(\tau))}{(t-\tau)^\mu}d\tau\,,\;h(t)\geq0\,,\;\mbox{ a.e. in }t\in(0,T)\,.\ee
where $\psi(t) \in L(\ssfrac{1}{\mu},\infty)(0,T)$ and $g$ is a function such that, a.e. in $\tau\in(0,T)$, \begin{equation}\label{PgGI}
g(\tau,h(\tau))\leq  A h(\tau)\hskip0.05cm |\tau-t_0|^{-\nu}+Bh^q(\tau)\hskip0.05cm|\tau-t_0|^{-\sigma}\,,  \text{ for }  q\geq1,
\end{equation}
with $A$, $B\geq 0$, independent of $\tau$, $t_0 \in [0,T)$ and for  suitable exponents $\nu,\sigma\in[0,1)$.\par
We recall Gronwall's inequalities similar to \rf{GRLNI} stated  in  \cite{DM} and \cite{GMS}  slightly different in the hypotheses and thesis. In \cite{DM}, the authors consider the following weakly singular Gronwall's inequality:
\begin{equation}\label{GIDM}
x(t)\leq \psi (t) + \int_0^t \sfrac{x(s)}{(t-s)^{\delta}}\,ds\quad \text{with}\,\, \delta \in (0,1).
\end{equation}
Instead, in \cite{GMS}, Giga {\it et al.} discuss a  weakly singular Gronwall's inequality of the type:
\begin{equation}\label{GIG}
f(t)\leq \alpha+ \int_0^t a(t,s)f(s)\,ds + \beta\int_0^t [1+ \log(1+f(s))]f(s)\,ds\, , 
\end{equation}
for $a(s,t)$ that satisfy suitable assumptions. Both  assume that the functions $x(t)$ and $f(t)$, respectively, are continuous on $(0,T)$. The thesis is a suitable bound of $x(t)$ or $f(t)$ by means of $\psi(t)$ or $\alpha$, respectively.\par  
Now, we turn to the discussion of \eqref{GRLNI}. The following result holds:
\begin{lemma}\label{L-GWSI}{\sl 
In \eqref{PgGI}, let $B=0$. Moreover, suppose that $\mu+\nu<1$ and   $|t-t_0|^{-\nu}h(t)\in L^1(0,T)$. Then:
\be\label{I-G-I}\Big[\intll0t h^s(\tau)d\tau\Big]^\frac1s\leq t^{\frac1s-\mu}\dm h\dm_{(\frac{1}{\mu},\infty)}\leq C \dm \psi \dm_{(\frac{1}{\mu},\infty)}\,\;\mbox{for all }s\in[1,\ssfrac1\mu)\mbox{ and }t\in [0,T)\,,   \ee 
where $C:= C(A,T)$ is a constant.
}
\end{lemma}
\bp 
Let $t \in (0,T)$.  At first, we discuss \rf{GRLNI} on $(0,\ov t)$, with $\ov{t}= \ssfrac{t}{N}$ and $N \in \N$ that we fix later.  We set $\ov{h}(t):=\chi_{(0,\ov{t})}h(t)$, where $\chi_{(0,\ov{t})}$ is the charatteristic function of $(0,\ov{t})$ and we consider the integral operator  $$ \mathbb T( g)(t):=\int_{\R}\sfrac{g(\tau,\ov{h}(\tau))}{|t-\tau|^\mu}d\tau\,,\mbox{ for all }t\in(0,\ov t)\,.$$ From the theory of weakly singular integral, we know that $\mathbb T$ is linear and continuous from $L^1(\R)$  into $L(\frac1\mu,\infty)(\R)$. Hence, the following estimates hold:
\be\label{G-II}\dm \mathbb T(g(\tau,\ov h(\tau))\dm_{(\frac1\mu,\hskip0.05cm\infty)}\leq cA \dm \ov{h}\hskip0.05cm|t_0-\tau|^{-\nu}\dm_1\,.\ee 
Taking \rf{GRLNI} into account, the inequality $\ov{h}(t)\leq \psi(t) + \mathbb T(g(\tau,\ov h(\tau))(t)$ holds almost everywhere in $t\in(0,\ov t)$. Hence, on $(0,\ov t)$  we  get
\be\label{G-III}\dm \ov{h}\dm_{(\frac1\mu,\hskip0.05cm\infty)}\leq 2\dm \psi \dm_{(\frac{1}{\mu},\infty)}+ 2\dm \mathbb T(g(\tau,\ov h(\tau))\dm_{(\frac1\mu,\hskip0.05cm\infty)}\leq 2\dm \psi \dm_{(\frac{1}{\mu},\infty)}+ c A\dm \ov{h}\hskip0.05cm|t_0-\tau|^{-\nu}\dm_1 \,.\ee 
Since in our assumptions $\mu+\nu<1$, the interval $(\frac1{1-\nu},\frac{1}{\mu})$ is not empty.  Applying H\"older's inequality, via estimate \rf{LQFLPW-I}, for all $s\in (\frac1{1-\nu},\frac{1}{\mu})$, we get
\be\label{G-IV}\dm \ov{h}\hskip0.05cm|t_0-\tau|^{-\nu}\dm_1\leq c\dm \ov{h}\dm_s\ov t\hskip0.05cm^{\frac1 {s'}-\nu}\leq c\ov t^{\hskip0.05cm \gamma}\dm \ov{h}\dm_{(\frac1\mu,\hskip0.05cm\infty)}\,,\mbox{ with }\gamma:=1-\mu-\nu>0\,.\ee
Increasing the right-hand side of \rf{G-III} by means of  \rf{G-IV}, we have
\[\dm \ov{h}\dm_{(\frac1\mu,\hskip0.05cm\infty)}\leq 2\dm \psi \dm_{(\frac{1}{\mu},\infty)}+cA\ov t^{\hskip0.05cm \gamma}\dm \ov{h}\dm_{(\frac1\mu,\hskip0.05cm\infty)}. \]
Choosing $N$ large enough to get $1-cA\ov t^{\hskip0.05cm \gamma}>0$ and moving $cA\ov t^{\hskip0.05cm \gamma}\dm \ov{h}\dm_{(\frac1\mu,\hskip0.05cm\infty)}$ to the left-hand side, we arrive at
\[\dm \ov{h}\dm_{(\frac1\mu,\hskip0.05cm\infty)}\leq \sfrac{1}{1-cA\ov t^{\hskip0.05cm \gamma}}\dm \psi \dm_{(\frac{1}{\mu},\infty)},\]
that, via estimate \eqref{LQFLPW-I} for $\dm h\dm_s$, implies the thesis on $(0,\ov{t})$. Now, recalling the definition of $\ov t$, let us consider the following partition:
$$(0,t)=\cup_{i=0}^{N-1}(i\hskip0.05cm\ov{t},(i+1)\hskip0.05cm\ov{t})$$
Defining $\ov{h}_i:= h \hskip0.05cm\chi_{(i\hskip0.05cm\ov{t},(i+1)\hskip0.05cm\ov{t})}$ and iterating the same arguments used before, we get:
\begin{equation*}
\dm h\dm_{L(\frac1\mu,\hskip0.05cm\infty)(0,t)}\leq c\mbox{$\underset{i=1}{\overset{N-1}\sum}$} \dm \ov{h}_i\dm_{L(\frac1\mu,\hskip0.05cm\infty)(i\hskip0.05cm\ov{t},(i+1)\hskip0.05cm\ov{t})}\leq \sfrac{c}{1-cA t^{\gamma}}\dm \psi \dm_{(\frac{1}{\mu},\infty)}.
\end{equation*}
Again, via estimate \eqref{LQFLPW-I}, last relation implies the thesis on $(0,t)$ for all $t \in [0,T)$. 
\end{proof}
\begin{lemma}\label{GRLNL}{\sl 
Let  $\psi \equiv 0$ in \eqref{GRLNI}. Let $\nu + \mu<1$, $\sigma + \mu<1$ and $q \in [1,\frac{1}{\mu})$.
Moreover, assume that
\[\sigma< 1-\ssfrac{q}{s}, \quad\text{with}\,\,s \in \big(\ssfrac{1}{1-\nu},\ssfrac{1}{\mu}\big)\,\, \text{and}\,\, s>q,\] 
and  $Ah|t_0-\tau|^{-\nu}+Bh^q|t_0-\tau|^{-\sigma}\in L^1(0,T)$. Then in \rf{GRLNI} $h$ is zero almost everywhere in $t\in(0,T)$\,.}\end{lemma} 
\bp
We initially discuss \rf{GRLNI} on $(0,\ov t)$ with $\ov t$ small enough. By iterating the arguments, the result trivially can be extended to all $\ov t\in (0,T)$ once again.  We set $\ov{h}(t):=\chi_{(0,\ov{t})}h(t)$ and consider  the integral operator again  $$ \mathbb T( g)(t):=\intl{\R}\sfrac{ g(\tau,\ov h)}{|t-\tau|^\mu\hskip-0.15cm}\hskip0.15cmd\tau\,,\mbox{ for all }t\in(0,\ov t)\,.$$ Operator $\mathbb T$ is linear and continuous from $L^1(\R)$  into $L(\frac1\mu, \infty)(\R)$. Hence, the following estimates hold:
\be\label{GR-II}\dm  \mathbb T( g(\tau,\ov{h}(\tau))\dm _{(\frac1\mu,\hskip0.05cm\infty)}\leq c \dm  g(\tau,\ov{h}(\tau))\dm _1\leq cA\dm  \ov{h}\hskip0.05cm|t_0-\tau|^{-\nu}\dm _1+cB \dm  \ov{h}^{\hskip0.05cm q}\hskip0.05cm |t_0-\tau|^{-\sigma}\dm _1 \,.\ee  
Taking \rf{GRLNI} into account, the inequality $\ov{h}(t)\leq \mathbb T(g(\tau,\ov{h}(\tau))(t)$ holds almost everywhere in $t\in(0,\ov t)$. Hence, on $(0,\ov t)$  we  get
\be\label{GRLN-III}\dm  \ov{h}\dm _{(\frac1\mu,\hskip0.05cm\infty)}\leq \dm  \mathbb T(g(\tau, \ov{h}(\tau))\dm _{(\frac1\mu,\hskip0.05cm\infty)}\leq  c \dm  g(\tau,\ov{h}(\tau))\dm _1\leq cA\dm  \ov{h}\hskip0.05cm|t_0-\tau|^{-\nu}\dm _1+cB \dm  \ov{h}^{\hskip0.05cm q} \hskip0.05cm|t_0-\tau|^{-\sigma}\dm _1 \,.\ee 
By virtue of \rf{LQFLPW-I} and \rf{GRLN-III}, for $s\in[1,\frac1\mu)$,  we also deduce
\be \label{hs}
\dm  \ov{h}\dm _s\leq c\hskip0.05cm\ov t^{\hskip0.05cm\frac1s-\mu}\Big[A\dm  \ov{h}\hskip0.05cm|t_0-\tau|^{-\nu}\dm _1+B \dm  \ov{h}^{\hskip0.05cm q} \hskip0.05cm|t_0-\tau|^{-\sigma}\dm _1\Big]\,.
\ee
Since $\nu+\mu<1$, we can consider $s\in(\frac1{1-\nu},\frac1\mu)$. Hence, $s'\nu<1$ holds. Applying  H\"older's inequality, we get
$$\dm  \ov{h}\hskip0.05cm|t_0-\tau|^{-\nu}\dm _1\leq c\hskip0.05cm\dm  \ov{h}\dm _s\hskip0.05cm\ov t^{\hskip0.05cm\frac1{s'\hskip-0.1cm}\hskip0.1cm-\nu}\,,\mbox{ for }s\in (\ssfrac1{1-\nu},\ssfrac1\mu)\,.$$ Hence, we also get
$$\dm  \ov{h}\hskip0.05cm|t_0-\tau|^{-\nu}\dm _1\leq c\dm  \ov{h}\dm _s\ov t^{\hskip0.05cm\frac1{s'\hskip-0.1cm}\hskip0.1cm-\nu}\leq c\ov t^{\hskip0.05cm1-\mu-\nu}\Big[A\dm  \ov{h}\hskip0.05cm|t_0-\tau|^{-\nu}\dm _1+B \dm  \ov{h}^{\hskip0.05cm q} \hskip0.05cm|t_0-\tau|^{-\sigma}\dm _1\Big]\,.$$
Therefore, summing we arrive at
$$\dm  \ov{h}\hskip0.05cm|t_0-\tau|^{-\nu}\dm _1+\dm  \ov{h}\dm _s\leq 
c\big[\ov t^{\hskip0.05cm1-\mu-\nu}+\ov t^{\hskip0.05cm\frac1s-\mu}\big]\Big[A\dm  \ov{h}\hskip0.05cm|t_0-\tau|^{-\nu}\dm _1+B \dm  \ov{h}^{\hskip0.05cm q} \hskip0.05cm|t_0-\tau|^{-\sigma}\dm _1\Big]\,.$$ By assumption on $\sigma$, setting $\theta=1-\frac{q}{s}-\sigma$, a further application of H\"older's inequality furnishes
\[\dm  \ov{h}^{\hskip0.05cm q} \hskip0.05cm|t_0-\tau|^{-\sigma}\dm _1\leq \bigg(\int_0^{\ov{t}}h(\tau)^s\, d\tau\bigg)^{\frac{q}{s}}\bigg(\int_0^{\ov{t}}\hskip0.05cm|t_0-\tau|^{-\sigma(\frac{s}{s-q})}\bigg)^{1-\frac{q}{s}}=c \,\ov{t}^{\hskip0.05cm \theta}\,\dm\ov{h}\dm_s^q \]
from which we arrive at
$$\dm  \ov{h}\hskip0.05cm|t_0-\tau|^{-\nu}\dm _1+\dm  \ov{h}\dm _s\leq 
c\big[\ov t^{\hskip0.05cm1-\mu-\nu}+\ov t^{\hskip0.05cm\frac1s-\mu}\big]\Big[A\dm  \ov{h}\hskip0.05cm|t_0-\tau|^{-\nu}\dm _1+B\ov{t}^{\hskip0.05cm \theta} \dm \ov{h}\dm _s^q\Big]\,.$$
Thanks to estimate \eqref{hs}, for $\ov t$ sufficiently small we get $\dm  h\dm _s\leq1$, and a further possible assumption of smallness on $\ov t$ one proves $h=0$ a.e. in $(0,\ov t)$. Hence, the lemma is proved.
 \ep 
\subsection{The Stokes Cauchy problem}
Let us consider the Stokes Cauchy problem
\begin{equation}\label{eq_calore}
\begin{cases}
u_t(t,x)-\Delta u(t,x)=-\n \pi(t,x) + \n \cdot (a \otimes b(t,x))\qquad &\text{in} \quad (0,\infty)\times \R^n\,, \\
\nabla \cdot u(t,x)=0 \qquad &\text{in} \quad (0,\infty)\times \R^n \,,\\
u(0,x)=u_0(x) \qquad &\text{on}\quad  \{0\}\times \R^n\,.
\end{cases}
\end{equation}
We recall the following results
\begin{lemma}\label{SHE}{\sl In \eqref{eq_calore}, assume that $a=b\equiv 0$. Let $\varphi_0\in \mathscr C_0(\R^n)$. Then we get a unique solution to problem \rf{eq_calore} such that
\be\label{RHE}\varphi(t, x)\in C^k(0,T;{\underset{q>1}\cap} J^q(\R^n))\cap({\underset{q>1}\cap} L^q(0,T;W^{2,q}(\R^n))\,,\mbox{ for }k\in\N_0\mbox{ and for all }T>0\,,\ee
and the representation formula holds:\be\label{RHEO} 
\varphi(t,x)=\int_{\R^n} H(t,x-y)\varphi_0(y)\,dy\,.
\ee
In particular, for all $r\geq q>1$, there exists a constant $c$ such that \be\label{SP}\dm \varphi(t)\dm_r+(t-s)^\frac12\dm\n \varphi(t)\dm_r\leq c(t-s)^{-\frac n2\left(\frac1q-\frac1r\right)}\dm \varphi(s)\dm_q\,,\mbox{ for all }t>s\geq0\,.\ee Finally,  for  $q\in(1,\infty)$ and $r\in[1,\infty)$, we get 
\be\label{SPLS}\ba{c}\varphi\in C([0,T);L(q,r)(\R^n))\,,\VS\dm \varphi(t)\dm_{(q,r)}+(t-s)^\frac12\dm\n \varphi(t)\dm_{(q,r)}\leq c\dm \varphi(s)\dm_{(q,r)}\,\mbox{ for all }t>s\geq0\,.\ea\ee
}\end{lemma}\bp The existence, the  uniqueness and properties \rf{RHE}-\rf{RHEO} are classical results and we refer to monograph \cite{LSU}. Concerning \rf{SP}, it also is classic, in any case is an easy consequence of the Young theorem applied to \rf{RHEO}. Property \rf{SPLS}$_1$    follows from \rf{RHE} via the interpolation between Lebesgue spaces. Finally, for estimate \rf{SPLS}$_2$ see, {\it e.g.}, \cite{Maremonti_lorentz,yamazaki}. \ep
With the aim to making the work self-contained, we give the existence theorem and uniqueness of problem \rf{eq_calore} with $a=b\equiv 0$ and 
with initial data in some weighted Lebesgue spaces and divergence free in weak sense. For a general initial boundary value problem in exterior domains in the weighted spaces, we refer to the papers \cite{Ko-Ku}, \cite{Fw-Sh}, \cite{Fr-I,Fr-II}. However, in the quoted papers, in the case of the Cauchy problem, we are not able to detect if estimate \rf{DCWS} holds with $c=1$, that is, the achievement of the same result as the usual $L^p$ theory.
 \par The following result holds:
\begin{theorem}\label{SPWS}{\sl 
In \eqref{eq_calore}, assume  $a=b\equiv 0$. Let $\alpha=1-\frac{n}{p}$ with $p \in (n,+\infty)$ and $n\geq 3$. Assume $u_0 \in J_\alpha^p(\R^n)$. Then, there exists a unique  smooth solution $u(t,x)$ to the Cauchy problem \eqref{eq_calore} such that 
\begin{equation}\label{DCWS}
\dm u(t)\dm _{\alpha, p}\leq \dm u(s)\dm _{\alpha,p}\,,\mbox{ for all }t>s \geq 0\,,
\end{equation}
\begin{equation}\label{convergenza_parte_lineare}
\lim_{t \to 0^+} \dm u(t) - u_0\dm _{\alpha,p} = 0\,.
\end{equation}
The  representation formula \rf{RHEO}  holds with $\varphi_0(y)=u_0(y)$.
Moreover, for all $s,h \in \N_0$ with $s+h \geq 0 $, and for $q \in (p,\infty)$, there is a costant $c$ such that 
\begin{equation}\label{gradienti_eq_stokes_beta_q}
\dm D_t^h\nabla^s u(t)\dm _{q}\leq ct^{-\frac{s}{2}-h-\frac{n}{2}({\frac{1}{n}}-\frac{1}{q})}\dm u_0\dm _{\alpha, p}\,, \text{ for all } t>0\,.
\end{equation}
  Finally,  for all $T>0$ and $q>p$, we get \be\label{CRHE}u\in C([0,T);L^p_\alpha(\R^n))\cap C((\vep,T); L^q(\R^n)),\mbox{ for all }\vep>0\,.\ee} 
\end{theorem}
\begin{proof}{[{\it Existence}]}
Firstly, assume that $\varphi_0 \in \mathscr{C}_0(\R^n)$. By virtue of Lemma\,\ref{SHE}, we know that exists one, and only one, smooth solution $(\varphi(t, x),c(t))$   to the equations \eqref{eq_calore} enjoying \rf{RHE}.    
Let $\alpha'=-\alpha$ and $p'$ conjugate exponent of $p$. Let $h_R \in [0,1]$ be a smooth cutoff function with:
\begin{equation*}
h_R =
\begin{cases}
1 \quad \text{if}\,\, |x|\leq R \,,\\
0 \quad \text{if}\,\, |x|\geq 2R\,,
\end{cases}
\end{equation*}
and $\nabla h_R=O(\frac{1}{R})$. Recalling that for a $\phi$ radial function the following holds
\[\Delta \phi(r)=\phi '' (r) + \frac{n-1}{r}\phi '(r)\,,\]
then, multiplying   the first equation in \eqref{eq_calore} by $h_R\,\varphi(\sigma^2+|\varphi|^2)^{\frac{p'-2}{2}}|x|^{\alpha'  p'}$, with $\sigma>0$, by means of an integration by parts over $\R^n$, we obtain
\be\label{int_parti_SE}\hskip-0.2cm\ba{ll}\displ
&\displ\sfrac{d}{dt}\dm h_R^\frac1{p'\hskip-0.1cm}(|\varphi(t)|^2+\sigma^2)^{\frac{1}{2}}\dm _{\alpha' , p'}^{p'}+ p'\!\intl{\R^n}\!h_R\,|\nabla \varphi(t,x)|^2(\sigma^2+|\varphi(t, x)|^2)^{\frac{p'-2}{2}}|x|^{\alpha' p'}\,dx   \\
&\displ\hskip3.1cm+ p'(p'-2)\!\intl{\R^n}h_R\Big(\nabla \varphi(t,x)\varphi(t,x) \Big)^2(\sigma^2+|\varphi(t, x)|^2)^{\frac{p'-4}{2}}|x|^{\alpha' p'} dx \\
&\displ\hskip3.6cm=-I_R(t,\sigma)+\alpha'  p'(\alpha'  p'\hskip-0.05cm +\hskip-0.05cm n \hskip-0.05cm- 2) \!\intl{\R^n}\!h_R(\sigma^2\hskip-0.1cm+|\varphi(t, x)|^2)^{\frac{p'-2}{2}}|x|^{\alpha'  p'-2}dx ,
\ea\ee where we set
$$I_R(t,\sigma):= \intl{\R^n}|x|^{\alpha'p'}(\sigma^2+|\varphi|^2)^\frac{p'}2 \Delta h_R dx+2 \intl{\R^n} \n h_R\cdot\n |x|^{\alpha'p'}(\sigma^2+|\varphi|^2)^{\frac{p'}{2}}dx\,.$$
Integrating both sides of  relation \rf{int_parti_SE} on $(0,t)$, recalling \eqref{RHE} for $q=p'$ and letting $\sigma \to 0$ first and subsequently $R\to \infty$, by virtue of the Lebesgue theorem of the dominate convergence we have:
\begin{align*}
&\dm \varphi(t)\dm _{\alpha' , p'}^{p'}+ p' \int_0^t\intl{\R^n}|\nabla \varphi(t,x)|^2|\varphi(t,x)|^{p'-2}|x|^{\alpha' p'}\,dx\,d\tau+ \\ 
& \hskip4cm+ p'(p'-2)\int_0^t\intl{\R^n}\Big(\nabla \varphi(t,x)\varphi(t,x) \Big)^2|\varphi(t,x)|^{p'-4}|x|^{\alpha' p'}\, dx\\
&\hskip6cm=\alpha'  p'(\alpha'  p' + n - 2)\int_0^t \intl{\R^n}|\varphi|^{p'}|x|^{\alpha'  p'-2}\,dx + \dm \varphi_0\dm _{\alpha',p'}.
\end{align*}
In our assumptions, the first term on the right hand in the last equation is a negative quantity. So, for all $t>0$, we can deduce that
\begin{equation}\label{stima_per_convergenza}
\dm \varphi(t)\dm_{\alpha',p'}\leq \dm\varphi_0\dm _{\alpha',p'}
\end{equation}
for $p' \in (1, \frac{n}{n-1})$ and for all $t>0$. Testing the equations in $\varphi(t, x)$, one easily achieves the weak continuity, that is $(\varphi(t),\phi)\in C((0,T))$, for all $\phi\in L^p_\alpha(\R^n)$. Hence, employing  estimate \rf{stima_per_convergenza}, one deduces 
\be\label{FTS}\lim_{t\to0}\dm \varphi(t)-\varphi_0\dm _{\alpha',p'}=0\,.\ee
  We consider
\be\label{FRSU}u(t,x)=\int_{\R^n} H(t,x-y)u_0(y)\,dy\ee
the solution to problem \eqref{eq_calore} with initial datum $u_0\in \mathscr C_0(\R^n)$. By the H\"older inequality and by estimate \eqref{stima_per_convergenza} we get
\be\label{DALQ}\ba{ll}
(u(t),\varphi_0)\hskip-0.2cm&=\!\displ \int_{\R^n}\!\!\Big[\int_{\R^n}\!\!\!H(t,x-y)u_0(y)\,dy\Big]\varphi_0(x)dx=\!\int_{\R^n}\!\!\Big[\int_{\R^n}\!\!\!H(t,x-y)\varphi_0(x)\,dx\Big]u_0(y)dy\VS
&=\displ \int_{\R^n}\varphi(t,y)u_0(y)dy\leq \dm \varphi(t)\dm _{\alpha',p'}\dm u_0\dm _{\alpha,p}\leq \dm \varphi_0\dm _{\alpha',p'}\dm u_0\dm _{\alpha,p}
\ea\ee
from which we obtain:
\begin{equation}\label{stima_alfa_p}
\dm u(t)\dm _{\alpha,p}\leq \dm u_0\dm _{\alpha,p}
\end{equation}
for $p>n$ and for all $t>0$.  By testing \rf{eq_calore} in $u(t, x)$, one deduces the weak continuity, and employing \rf{stima_alfa_p} we get 
    $$\lim_{t\to0}\dm u(t)-u_0\dm_{p,\alpha}=0\,.$$ Differentiating  \rf{FRSU}, then, employing the duality argument used in \rf{DALQ} for $u_t$, for all $t>0$, we also get
\be\label{RGE-I}\dm u_t(t)\dm_{\alpha,p}\leq \dm \Delta u_0\dm_{\alpha,p}\,.\ee
Hence, as a consequence of estimates \rf{stima_alfa_p} and \rf{RGE-I} and regularity of $u(t, x)$, we get \be\label{RGE-III}u\in C([0,T);L^{p}_{\alpha}(\R^n))\,,\mbox{ for all }T>0\,.\ee
{{Now, we need to estimate $D_t^h \nabla_x^s u(t,x)$. Firstly, we observe that by virtue of Lemma \ref{IWSLq}, estimates \eqref{FGMSI} and \eqref{stima_alfa_p} , for all $q \in (p,\infty)$ and $t>0$, we have:
\be\label{SLQ}\dm u(t)\dm_{q}\leq c \dm u(t) \dm^{1-\gamma}_{\alpha,p} \dm u(t)\dm^{\gamma}_{\infty}\leq c t^{-\frac{n}{2}(\frac{1}{n}-\frac{1}{q})}\dm u_0\dm_{\alpha,p}\ee
Now, recalling the well known estimate 
\begin{equation} \label{usare_per_conv_2}
\dm D_t^h\nabla^s u(t-\sigma)\dm _{q}\leq c (t-\sigma)^{-\frac{s}{2}-h} \dm u(\sigma)\dm_q\,,  
\end{equation}
for all $h,s \in \N_0$ and $t>\sigma\geq0$,}} choosing $\sigma=\frac t2$ and then employing \rf{SLQ}, we arrive at 
\be\label{LQ-LW}\dm D_t^h\nabla^s u(t)\dm _{q}\leq ct^{-\frac{s}{2}-h-\frac{n}{2}(\frac{1}{n}-\frac{1}{q})}\dm u_0\dm _{\alpha, p}\,.\ee 
Now, we consider $u_0\in J_{\alpha}^p(\R^n)$. There exists a sequence $\{u_0^m\}\subset \mathscr{C}_0$ such that $u_0^m \to u_0$ in $J_{\alpha}^p(R^n)$ and we indicate with $\{u^m(t,x)\}$ the sequence of smooth solutions to the Cauchy problem \rf{eq_calore} with $a=b=0$.
By virtue of the linearity of the problem  and \eqref{LQ-LW},  we also  get:
\begin{gather*}
\dm u^m(t)-u^{\nu}(t)\dm _{\alpha, p}\leq \dm u_0^m-u_0^{\nu}\dm _{\alpha,p}\\
\dm D_t^h\nabla^s u^m(t)-D_t^h\nabla^s u^{\nu}(t)\dm _{q}\leq ct^{-\frac{s}{2}-h-\frac{n}{2}({\frac{1}{n}}-\frac{1}{q})}\dm u^m_0-u^{\nu}_0\dm _{\alpha, p}
\end{gather*}
which ensure existence of a limit $u(t,x)$ uniformly in $t>0$ and $u$ enjoys \rf{CRHE}. Moreover, for the limit $u(t,x)$, \eqref{stima_alfa_p} and \eqref{usare_per_conv_2} trivially hold.\par[{\it Uniqueness}] As it is usual for the $L^p$-theory, also in our special $L^p$-spaces,  we use a duality argument to state   the uniqueness. Let  $(u(t,x), \pi_u(t,x))$ and $(v (t,x), \pi_v(t,x))$ be two solutions to the Cauchy Stokes problem \eqref{eq_calore} with initial datum $u_0$. We set $w(t,x): = u (t,x)-v (t,x)$ and $\pi_w(t, x):=\pi_u(t, x)-\pi_v(t, x)$. So, $w$ is a solution to following problem:
\begin{equation} \label{eq_w}
\begin{cases} 
w_t - \Delta w = -\n\pi_w\qquad &\text{in} \quad (0,\infty)\times \R^n \\
\n \cdot w(t,x)=0 \qquad &\text{in} \quad (0,\infty)\times \R^n \\
w(0,x)=0\qquad &\text{on}\quad  \{0\}\times \R^n
\end{cases}
\end{equation}
Let $\varphi_0 \in \mathscr{C}_0(\R^n)$ 
and let $(\varphi(t,x),c(t))$ a solution to the Stokes Cauchy problem \eqref{eq_calore} corresponding to $\varphi_0$ enjoying \rf{RHE}, \rf{stima_per_convergenza} and  \rf{FTS}. We define:
\begin{equation}
\widehat{\varphi}(\tau,x) = \varphi(t-\tau,x) \qquad \text{for} \,(\tau,x) \in (0,t)\times \R^n
\end{equation}
As it is known $\widehat{\varphi}$ is a solution backward in time on $(0,t)\times\R^3$ with $
\widehat{\varphi}(t,x) = {\varphi}_0(x) $ on $  \{t\}\times \R^n
$.
Multiplying the first equation in  \eqref{eq_w} for $\widehat{\varphi}$ and integranting by parts on $[s,t] \times \R^n$, we obtain:
\begin{equation*}
(w(t),\varphi_0) = (w(s),\varphi(t-s))
\end{equation*}
from which 
\begin{equation*}
|(w(t),\varphi_0)| = |(w(s),\varphi(t-s))| \leq \dm w(s)\dm _{\alpha,p}\dm \varphi(t-s)\dm _{\alpha',p'}
\end{equation*}
and, considering \eqref{stima_per_convergenza},
we get 
\begin{equation*}
|(w(t),\varphi_0)| \leq \dm w(s)\dm _{\alpha,p} \dm \varphi_0\dm _{\alpha',p'}
\end{equation*}
The second term tends to $0$ for $s \to 0$, so $$(w(t),\varphi_0) = 0$$Due to the arbitrariness of $\varphi_0$ and $t$, the last relation implies that $w = 0 $ on $\{t\} \times \R^n$ for all  $t>0$.
This complete the proof. \end{proof}
\begin{lemma}\label{GMSI}{\sl
Let $g\in L^p_{\alpha}(\R^n)$. Then for the convolution product $H\ast g$ the following estimate holds
\begin{equation}\label{FGMSI}
t^{\frac{1}{2}}\dm H\ast g(t)\dm_{\infty}\leq c \dm g\dm_{\alpha,p}, \quad \text{for all}\,\, t>0.
\end{equation}}
\end{lemma}
\begin{proof}
See, {\it e.g.}, Lemma\,6 in \cite{GM}.
\end{proof}
\begin{lemma}\label{GMWP}{\sl Let $g\in L^p_\alpha(\R^n)$. There exists a constant $c$, independent of $g$, such that for all $(t,x)$
\be\label{GMWP-I}t^{\frac 12(1-\alpha)} |x|^\alpha|H*g (t, x)|\leq c\dm g\dm_{\alpha,p}\,.\ee}\end{lemma}
\bp This result is proved in \cite{GM}, Corollary\,1. Here, for the convenience of the reader, we reproduce the very short proof. Applying H\"older's inequality, we get
$$|H*g(t,x)|\leq  \big[ H^{p'}*|y|^{-\alpha p'}(t,x)\big]^\frac1{p'}\dm g\dm_{\alpha,p}\,.$$
Recalling \rf{EHE}, $|H(t,z)|\leq c(|z|+t^\frac12)^{-n}$, we show the following estimate\be\label{PELD}\ba{ll}H^{p'}*|y|^{-\alpha p'}(t,x)\hskip-0.2cm&\leq\displ \intl{|y|<\frac{|x|}2}H^{p'}(t,x-y)|y|^{-\alpha p'}dy+\intl{|y|>\frac{|x|}2}H^{p'}(t,x-y)|y|^{-\alpha p'}dy\VSE \leq c|x|^{-n}t^{-\frac n2\frac1{p-1}}\intl{|y|<\frac{|x|}2}|y|^{-\alpha p'}dy
+c|x|^{-\alpha p'}\intl{\R^n}H^{p'}(t,z)dz\VSE \leq c |x|^{-\alpha p'}t^{-\frac n2\frac1{p-1}}\,,\ea\ee where $c$ is a constant independent of $(t,x)$, which leads to estimate \rf{GMWP-I}.\ep
{The following results give an estimate that slightly differs from that obtained in  Lemma\,\ref{GMSI} and Lemma\,\ref{GMWP}. They are closely related to our approach. Indeed, to prove the existence of a solution to problem \eqref{problem}, we seek an upper bound that is appropriately small as a function of $t$. The following estimates will be utilized in Lemma \ref{Lemma_8} with $g=u_0$. }
\begin{lemma}\label{stima_L_infinito}{\sl
Let $g \in L^p_{\alpha}(\R^n)$. Then, for the convolution product $H\ast g$ the following estimate holds
\begin{equation}\label{stima_L_infinito_formula}
t^{\frac{1}{2}}\dm H\ast g(t)\dm _{\infty}\leq K_g(t,\rho),\quad \text{for all}\, \, t>0\mbox{ and }\rho>0\,.
\end{equation}
where $K_g(t, \rho)$ is definied in \eqref{K(t,ro)}.
}
\end{lemma}
\begin{proof}Let $\rho > 0$. We set 
\[\ov{H}(t,x):=\sfrac{1}{(4\pi t)^{\frac{n}{2}}}e^{-\frac{|x|^2}{8t}}\,.\]
Applying the H\"older inequality, we get:
\begin{align*}
|H\ast g(t,x)|&\leq\sfrac{c_0}{t^{\frac{n}{2p}}}\bigg[\int_{B_\rho(x)}\hskip-0.3cm \sfrac{H(t,x-y)}{|y|^{\alpha p'}}\,dy\bigg]^{\frac{1}{p'}}\dm g\dm _{\alpha,p,\rho}+ c_1\,\sfrac{e^{-\frac{\rho ^2}{8t}}}{t^{\frac{n}{2p}}}\bigg[\int_{|x-y|>\rho} \hskip-0.3cm\sfrac{\ov H(t,x-y)}{|y|^{\alpha p'}}\,dy\bigg]^{\frac{1}{p'}}\dm g\dm _{\alpha,p}\\&\leq \sfrac{c_0}{t^{\frac{n}{2p}}}\bigg[\int_{\R^n} \sfrac{H(t,x-y)}{|y|^{\alpha p'}}\,dy\bigg]^{\frac{1}{p'}}\dm g\dm _{\alpha,p,\rho}+ c_1\,\frac{e^{-\frac{\rho ^2}{8t}}}{t^{\frac{n}{2p}}}\bigg[\int_{\R^n} \sfrac{\ov H(t,x-y)}{|y|^{\alpha p'}}\,dy\bigg]^{\frac{1}{p'}}\dm g\dm _{\alpha,p}
\end{align*}
for all $(t,x)\in (0,T)\times \R^n$.\\
The functions
\begin{equation*}
\mathcal{G}(t,x):= \int_{\R^n} \sfrac{H(t,x-y)}{|y|^{\alpha p'}}\,dy\mbox{\; and\; }\ov{\mathcal{G}}(t,x):= \int_{\R^n} \sfrac{\ov H(t,x-y)}{|y|^{\alpha p'}}\,dy
\end{equation*}
are the solutions to the heat equation corresponding to the initial datum $|y|^{-\alpha p'}\in L(\frac{n}{\alpha p'},\infty)(\R^n)$. By the properties of the solution to the heat equation in Lorentz spaces (see, {\it e.g.},  \cite{Maremonti_lorentz} and \cite{yamazaki}) we deduce that
\begin{equation}\label{sol_calore_stima_weak}
\dm \mathcal{G}(t)\dm _{\infty}\leq c t^{-\frac{\alpha}{2}p'}\mbox{\; and \;}\dm \ov{\mathcal{G}}(t)\dm _{\infty}\leq c t^{-\frac{\alpha}{2}p'}, \quad \text{for all}\,\, t>0
\end{equation}
So, we get:
\begin{equation*}
t^{\frac{1}{2}}|H \ast g(t,x)|\leq c_0 \dm g\dm _{\alpha,p,\rho} + c_1e^{-\frac{\rho ^2}{8t}}\dm g\dm _{\alpha,p} \,,\mbox{ for all }(t, x)\in (0,T)\times\R^n\,,
\end{equation*}
with $c_0$ and $c_1$ positive constants independent of $t$, $\rho$ and $g$. Hence, via definition \rf{K(t,ro)},  estimate \eqref{stima_L_infinito_formula} holds.
\end{proof}
\begin{lemma}\label{PPLF}{\sl
Let $g\in L^{p}_{\alpha}(\R^n)$. Then
\begin{equation}
t^{\frac{n}{2p}}|x|^{\alpha}|H\ast g(t,x)|\leq K_g(t,\rho)\,,\mbox{ for all }(t,x)\in(0,T)\times\R^n\,, 
\end{equation}
where $K_g(t, \rho)$ is definied in \eqref{K(t,ro)}.}
\end{lemma}
\begin{proof}
Let $\rho > 0$. We set 
\[\ov{H}(t,x):=\sfrac{1}{(4\pi t)^{\frac{n}{2}}}e^{-\frac{|x|^2}{8t}}\,.\]
For all $(t, x)$, we get
\begin{align*}
|H\ast g(t,x)|&\leq \int_{\R^n}H(t,x-y)|g(y)|\, dy \\&\leq \int_{|x-y|<\rho}\sfrac{H(t,x-y)}{|y|^{\alpha}}|g(y)||y|^{\alpha}\,dy 
+e^{-\frac{\rho^2}{8t}}\int_{|x-y|>\rho}\sfrac{\ov{H}(t,x-y)}{|y|^{\alpha}}|g(y)||y|^{\alpha}\,dy\\& \leq c_0 \,\Big[\int_{\R^n}\sfrac{H(t,x-y)^{p'}}{|y|^{\alpha p'}}\Big]^{\frac{1}{p'}}\dm g\dm_{\alpha,p,\rho} + c_1\, e^{\frac{\rho^2}{8t}}\Big[\int_{\R^n}\sfrac{\ov H(t,x-y)^{p'}}{|y|^{\alpha p'}}\Big]^{\frac{1}{p'}}\dm g\dm_{\alpha,p}\,.
\end{align*}
Taking estimate \rf{PELD} into account, we arrive at
\begin{align*}
|x|^{\alpha}\,t^{\frac{n}{2p}}\,|H\ast g(t,x)|\leq K_g(t,\rho)\,,\mbox{ for all }(t, x)\in (0,T)\times\R^n\,,
\end{align*}
that is the thesis.
\end{proof}
Now, in \eqref{eq_calore}, let us consider $u_0(x)\equiv 0$ and $a(t,x)$ and $b(t,x)$ divergence free vector function in a weak sense.\par
It's well known that the solution $v(t,x)$ of this problem has the following representation:
\be\label{OEAB}
v(t,x)=\int_0^t\int_{\R^n}\n E(t-\tau,x-y)a \otimes b(\tau,y)\,dy\,d\tau= \n E \ast a\otimes b (t,x)
\ee
Let us provide some a priori estimates:
\begin{lemma}\label{PPNLF}
{\sl
Assume  in \rf{OEAB}
\[\sup_{(0,t)}\tau^{\frac{1}{2}}\dm a\dm_{\infty}+\sup_{(0,t)}\tau^{\frac{n}{2p}}\sup_{\R^n}|y|^{\alpha}|b(\tau,y)|< \infty.\]
Then there exists a constant $c$ independent of $a$ and $b$ such that
\begin{equation}
|x|^{\alpha}\,t^{\frac{n}{2p}}|\n E \ast a\otimes b(t,x)|\leq c\,\sup_{(0,t)}\tau^{\frac{1}{2}}\dm a\dm_{\infty}\,\sup_{(0,t)}\tau^{\frac{n}{2p}}\sup_{\R^n}|y|^{\alpha}|b(\tau,y)|\,.
\end{equation}
}
\end{lemma}
\begin{proof}
We operate the following decomposition  of $\R^n$
\[\R^n=\big\{\,y\,:\,|y|\leq\sfrac{|x|}{2}\big\} \cup \big\{\,y\,:\,|y|>\sfrac{|x|}{2}\big\}\,.\]
Thus, we have
\begin{align*}
|\n E \ast a\otimes b(t,x)|&\leq c\, \int_0^t\int_{\R^n}\sfrac{|a(\tau,y)||b(\tau,y)|}{(|x-y|+(t-\tau)^{\frac{1}{2}})^{n+1}}\,dy\,d\tau\\&= c\!\int_0^t\!\int_{|y|<\frac{|x|}{2}}\!\!\sfrac{|a(\tau,y)||b(\tau,y)|}{(|x-y|+(t-\tau)^{\frac{1}{2}})^{n+1}}\,dy\,d\tau
+  c\!\int_0^t\!\int_{|y|>\frac{|x|}{2}}\!\!\sfrac{|a(\tau,y)||b(\tau,y)|}{(|x-y|+(t-\tau)^{\frac{1}{2}})^{n+1}}\,dy\,d\tau\\&=:c\,(I_1+I_2)\,.
\end{align*}
We set
\[D:=\sup_{(0,t)}\tau^{\frac{1}{2}}\dm a\dm_{\infty}\,\sup_{(0,t)}\tau^{\frac{n}{2p}}\sup_{\R^n}|y|^{\alpha}|b(\tau,y)|\,.\] 
Since $|x-y|\geq |x| -|y|\geq \frac{|x|}{2}$ for $|y|<\frac{|x|}{2}$, then  for $I_1$ we get
\begin{align*}
I_1\leq {D}{|x|^{-n}}\int_0^t \tau^{-\frac{1}{2}-\frac{n}{2p}}(t-\tau)^{-\frac{1}{2}}\int_{|y|<\frac{|x|}{2}}{|y|^{-\alpha}}\,dy\,d\tau\leq {\ov{c}\,D}{|x|^{-\alpha}\, t^{-\frac{n}{2p}}}\,.
\end{align*}
For $I_2$ we have
\begin{align*}
I_2 &\leq {\ov{c}\, D}{|x|^{-\alpha}}\int_0^t {\tau^{-\frac{1}{2}-\frac{n}{2p}}}\bigg(\int_{|y|>\frac{|x|}{2}}{(|x-y|+(t-\tau)^{\frac{1}{2}})^{-n-1}}\,dy\bigg)\,d\tau\VSE\leq {\ov{c}\, D}{|x|^{-\alpha}}\int_0^t {\tau^{-\frac{1}{2}-\frac{n}{2p}}(t-\tau)^{-\frac{1}{2}}}\,d\tau\leq {\ov{c}\, D}{|x|^{-\alpha}\, t^{-\frac{n}{2p}}}\,.
\end{align*}
Recalling the expression of $D$, we arrive at
\begin{align*}
|x|^{\alpha}\,t^{\frac{n}{2p}}|\n E \ast a\otimes b(t,x)|\leq c\,\sup_{(0,t)}\tau^{\frac{1}{2}}\dm a\dm_{\infty}\cdot\sup_{(0,t)}\tau^{\frac{n}{2p}}\sup_{\R^n}|y|^{\alpha}|b(\tau,y)|
\end{align*}
that is the thesis.
\end{proof}
\begin{lemma}\label{stim_termine_convettivo_inf}
{\sl Assume in \rf{OEAB}\begin{equation*}
\sup_{(0,T)}t^{\frac{1}{2}}\dm a(t)\dm _{\infty}<\infty,\quad \sup_{(0,T)}t^{\frac{n}{2}(\frac{1}{n}-\frac{1}{q})}\dm b(t)\dm _{q}<\infty
\end{equation*}
Then, there exists a constant $c>0$, independent of $a(t,x)$ and $b(t,x)$, such that
\begin{equation} \label{stima_lemma_convettivo_inf}
t^{\frac{1}{2}}\dm \nabla E \ast a \otimes b\dm _{\infty}\leq c\,\sup_{(0,t)} \tau^{\frac{1}{2}}\dm a(\tau)\dm _{\infty}\cdot \sup_{(0,t)}\tau^{\frac{n}{2}(\frac{1}{n}-\frac{1}{q})}\dm b(\tau)\dm _q
\end{equation}
for all $t\in(0,T)$.}
\end{lemma}
\begin{proof}
The proof is analogous to that contained in \cite{CM-L3} with a trivial adaptation to our setting.
\end{proof}
\begin{lemma}\label{stima_termine_convettivo_q}
{\sl Assume in \rf{OEAB} $$\displaystyle \sup_{(0,T)}t^{\frac{n}{2}(\frac{1}{n}-\frac{1}{q})}\dm a(t)\dm _{q} +\sup_{(0,T)}t^{\frac{n}{2}(\frac{1}{n}-\frac{1}{q})}\dm b(t)\dm _{q}<\infty\,.$$ Then  there exists a constant $c>0$, independent of $a(t,x)$ and $b(t,x)$, such that\begin{equation}
t^{\frac{n}{2}(\frac{1}{n}-\frac{1}{q})}\dm \nabla E \ast a \otimes b(t)\dm _q\leq c\, \sup_{(0,t)}\tau^{n(\frac{1}{n}-\frac{1}{q})}\dm a(\tau)\dm _q\dm b(\tau)\dm _q
\end{equation}
for all $t\in (0,T)$.}
\begin{proof}
By the Young inequality we have:
\begin{align*}
 \dm \nabla E \ast a \otimes b(t,x)\dm _q
&\leq  \int_0^t \dm \nabla E(t-\tau)\dm _{q'}\dm |a(\tau)||b(\tau)|\dm _{\frac{q
}{2}}\, d\tau\\&\leq  \sup_{(0,t)}\tau^{n(\frac{1}{n}-\frac{1}{q})}\dm |a(\tau)||b(\tau)|\dm _{\frac{q
}{2}}\int_0^t \sfrac{\dm \nabla E(t-\tau)\dm _{q'}}{\tau ^{n(\frac{1}{n}-\frac{1}{q})}}\, d\tau \\
& \leq c\,\sup_{(0,t)}\tau^{n(\frac{1}{n}-\frac{1}{q})}\dm |a(\tau)||b(\tau)|\dm _{\frac{q
}{2}} \int_0^t {\tau^{-n(\frac{1}{n}+\frac{1}{q})}(t-\tau)^{-\frac{1}{2}-\frac{n}{2q}}}\, d\tau\\&\leq c\, t^{-\frac{n}{2}(\frac{1}{n}-\frac{1}{q})}\sup_{(0,t)}\tau^{n(\frac{1}{n}-\frac{1}{q})}\dm a(\tau)\dm _q\dm b(\tau)\dm _q
 \end{align*}
which implies the thesis.
\end{proof}
\end{lemma}
\section{Proof of the main results}
For $m\geq 1$, we consider the following approximation scheme:
\begin{equation}\label{definizione_schema_iterativo}
u^m(t,x)=H\ast u_0(t,x)-\nabla_x E\ast (u^{m-1}\otimes u^{m-1})(t,x)
\end{equation}
with $u^0(t,x):= H\ast u_0$. We set
\begin{equation}\label{definizione_funzionale}
\widetilde{\tm} u(t)\tm:= \sup_{(0,t)} \tau^{\frac{1}{2}}\dm u(\tau)\dm _{\infty} + \sup_{(0,t)}\tau^{\frac{n}{2}(\frac{1}{n}-\frac{1}{q})}\dm u(\tau)\dm _q+\sup_{(0,t)}\sup_{\R^n}\tau^{\frac{n}{2p}}|y|^{\alpha}|u(\tau,y)|
\end{equation}
{and in the following we just write $K(t,\rho)$ in place of $K_{u_0}(t,\rho)$ defined in \rf{K(t,ro)}. We have}
\begin{lemma}\label{Lemma_8}{\sl 
Let $u_0 \in L^p_{\alpha}(\R^n)$. Then, there exist a constant $c>0$, independent of $u_0$, and $m \in \N$, such that for sequence \eqref{definizione_schema_iterativo}, for all $T>0$, we have
\begin{equation}\label{c}
\widetilde{\tm} u^m(t)\tm\leq K(t,\rho) + c\,\dm u_0\dm _{\alpha,p}^{1-\gamma}K(t,\rho)^{\gamma}+ c_2\,\widetilde{\tm} u^{m-1}(t)\tm\,,\mbox{ for all }t\in(0,T)\,.
\end{equation}}
\end{lemma}
\begin{proof}
Let $m=1$. From definition \eqref{definizione_schema_iterativo}, by virtue of Lemmas\,\ref{stima_L_infinito} and Lemma\,\ref{stim_termine_convettivo_inf}, for all $s>0$ and $\rho>0$, we get
\begin{align*}
s^{\frac{1}{2}}\dm u^1(s)\dm _{\infty}&\leq s^{\frac{1}{2}}\dm u^0 (s)\dm _{\infty} + c\,\Big(\sup_{(0,s)} \tau^{\frac{1}{2}}\dm u^0(\tau)\dm _{\infty}+\sup_{(0,s)}\tau^{\frac{n}{2}(\frac{1}{n}-\frac{1}{q})}\dm u^0(\tau)\dm _q\Big)^2 \\
&\leq K(s,\rho) + c\widetilde{\tm} u^0(s)\tm^2\,.
\end{align*}
Hence, being the right-hand side   an increasing function of $s$, on any interval $(0,t)$, we get 
$$\sup_{(0,t)} s^\frac12\dm u^1(s)\dm_\infty\leq K(t,\rho)+c\widetilde{\tm}u^0(t)\tm^2\,.$$
Applying   estimate \eqref{interpolazione} with $r=\infty$ for $u^0$ and Lemma\,\ref{stima_termine_convettivo_q} for the convective term, we have:
\begin{align*}
s^{\frac{n}{2}(\frac{1}{n}-\frac{1}{q})}\dm u^1(s)\dm _q &\leq s^{\frac{n}{2}(\frac{1}{n}-\frac{1}{q})}\dm u^0(s)\dm _q + c\,\Big( \sup_{(0,s)}\tau^{\frac{n}{2}(\frac{1}{n}-\frac{1}{q})}\dm u^0(\tau)\dm _q\Big)^2\\
& \leq c\dm u_0\dm _{\alpha,p}^{1-\gamma}K(s,\rho)^{\gamma} + c\,\widetilde{\tm} u^0(s)\tm^2\,.
\end{align*}
The same argument lines employed for the estimate in $L^\infty$ lead to the following:
$$\sup_{ (0,t)}s^{\frac{n}{2}(\frac{1}{n}-\frac{1}{q})}\dm u^1(s)\dm _q  \leq c\dm u_0\dm _{\alpha,p}^{1-\gamma}K(t,\rho)^{\gamma} + c\,\widetilde{\tm} u^0(t)\tm^2\,.$$ Finally, 
taking in account Lemma \ref{PPLF} for the linear part and Lemma \ref{PPNLF} for the one involving Oseen tensor, we get
\begin{align*}
s^{\frac{n}{2p}}|x|^{\alpha}|u^1(s,x)|\leq s^{\frac{n}{2p}}|x|^{\alpha}|u^0(s,x)| + c\sup_{(0,s)}\tau^{\frac{1}{2}}\dm u^0(\tau)\dm_{\infty}\sup_{(0,s)}\tau^{\frac{n}{2p}}\dm|y|^{\alpha}u^0(\tau)\dm_\infty\leq K(s,\rho) + c \,\widetilde{\tm} u^0(s) \tm\,.
\end{align*}Hence, on $(0,t)$, we get
$$\sup_{(0,t)}s^{\frac{n}{2p}}|x|^{\alpha}|u^1(s,x)|\leq K(t,\rho) + c \,\widetilde{\tm} u^0(t) \tm\,.$$
Summing last relations, we obtain
\begin{equation}\label{approssimazione_funzionale}
\widetilde{\tm} u^1(s)\tm\leq K(s,\rho)+  c\,\dm u_0\dm _{\alpha,p}^{1-\gamma}K(s,\rho)^{\gamma} + 3c\, \widetilde{\tm} u^0(s)\tm^2
\end{equation}
with constant $c$ independent of the datum $u_0$.   Now, let $m>1$ and let assume that $u^m$ satisfies the property \eqref{c}. We show that for $u^{m+1}$ estimate \eqref{c} still holds. Indeed, by inductive hypothesis, $u^m$ satisfies the assumptions of the Lemmas \ref{stima_L_infinito}-\ref{stima_termine_convettivo_q}. So, by the same computations done in the case $m=1$, we get
$$ s^{\frac{1}{2}}\dm u^{m+1}(s)\dm _{\infty}\leq K(s,\rho) + c\widetilde{\tm} u^{m}(s)\tm ^2$$
\[s^{\frac{n}{2}(\frac{1}{n}-\frac{1}{q})}\dm u^{m+1}(s)\dm _q\leq c\dm u_0\dm _{\alpha,p}^{1-\gamma}K(s,\rho)^{\gamma} +\,c\widetilde\tm u^m(s)\tm ^2\]
and
\[s^{\frac{n}{2p}}\,|x|^{\alpha}|u^{m+}(s,x)|\leq K(s,\rho) + c \,\widetilde\tm u^m (s)\tm\]
Considering of the last estimates the $\sup$ on $(0,t)$ and summing, we obtain
\[\widetilde\tm u^{m+1}(t)\tm \leq K(t,\rho)+  c\,\dm u_0\dm _{\alpha,p}^{1-\gamma}K(t,\rho)^{\gamma} + 3c\, \widetilde\tm u^m(t)\tm ^2\]
Therefore, \eqref{c} holds for all $m\in \N$. 
\end{proof}
\begin{lemma}\label{lemma_iterate}{\sl 
Let $\xi_0 >0$ and $c>0$. Let $\{\xi_m\}$ be a nonnegative sequence of real numbers such that:
\begin{equation*}
\xi_m\leq \xi_0+c\,\xi_{m-1}^2
\end{equation*}
Assume $1-4\xi_0>0$ and  $\xi_0\leq \xi$, where $\xi$ is the minimum solution of algebraic equation $c\xi^2-\xi+\xi_0=0$. Then $\xi_{m-1}\leq \xi$ for all $m\in \N$.}
\end{lemma}
\begin{proof}
For the proof we remind to \cite{solo}, Lemma 10.2.
\end{proof}
\begin{lemma}\label{Lemma_10}{\sl 
Let $\{u^m\}$ be the sequence defined by \eqref{definizione_schema_iterativo} corrisponding to $u_0 \in L^{p}_{\alpha}$. Then there exists $T(u_0)>0$ such that, for all $\eta$, the sequence strongly converges in $\mathit{C}((\eta,T(u_0)\times \R^n)$ to a solution $u$ of \eqref{equazione_integrale}. In particular, for a suitable $\rho$ and for all $t\in [0,T(u_0))$, we get:
\begin{equation}\label{stima_sol_A(ro,t)}
\widetilde\tm u(t)\tm \leq \sfrac{\left[K(t,\rho)+  c\,\dm u_0\dm _{\alpha,p}^{1-\gamma}K(t,\rho)^{\gamma}\right]}{1+\left[1-4c_2\,\left(K(t,\rho)+  c\,\dm u_0\dm _{\alpha,p}^{1-\gamma}K(t,\rho)^{\gamma}\right)\right]^{\frac{
1}{2}}}\leq \sfrac{1}{4c_2}\,,
\end{equation}
and, in particular
\begin{equation}\label{proprieta_limite}
\lim_{t \to 0^+}\tm u(t)\tm =0,
\end{equation}
\begin{equation}\label{andamento_in_q}
t^{\frac{1}{2}}\dm u(t)\dm _{\infty}+t^{\frac{n}{2}(\frac{1}{n}-\frac{1}{q})}\dm u(t)\dm _q + t^{\frac{n}{2p}} \sup_{\R^n}|y|^{\alpha}|u(\tau,y)|\leq K(t,\rho)+  c\dm u_0\dm _{\alpha,p}^{1-\gamma}K(t,\rho)^{\gamma}
\end{equation}}
\end{lemma}
\begin{proof}
Since $u_0 \in L^p_{\alpha}(\R^n)$, for any $\varepsilon \in (0,\frac{1}{4c_2\, c_0})$, there exists $\rho=\rho(u_0,\varepsilon)$ such that $\dm u_0\dm _{\alpha,p,\rho}<\varepsilon$. For any such $\rho$ we consider:
\begin{equation}\label{istante_sup_esistenza}
t(\rho):=\sup\big\{\,t>0\,\big|\,\big[1-4c_2\big(K(t,\rho)+  c\dm u_0\dm _{\alpha,p}^{1-\gamma}K(t,\rho)^{\gamma}\big)\big]>0 \big\}
\end{equation}
and we denote by $\displaystyle T(u_0):=\sup_{\rho>0}t(\rho)$. Then, thanks to \eqref{approssimazione_funzionale} and Lemma \ref{lemma_iterate}, for a fixed $\rho$ and for any $t\in [0,T(u_0))$, uniformly in $m\in \N$, we get:
\begin{equation}\label{A(ro,t)}
\widetilde\tm u^m(t)\tm \leq \sfrac{2\left[K(t,\rho)+  c\dm u_0\dm _{\alpha,p}^{1-\gamma}K(t,\rho)^{\gamma}\right]}{1+\left[1-4c_2\,\left(K(t,\rho)+  c\dm u_0\dm _{\alpha,p}^{1-\gamma}K(t,\rho)^{\gamma}\right)\right]^{\frac{
1}{2}}}=: A(\rho,t)
\end{equation}
This estimate ensures that for all $t\in [0,T(u_0))$, the sequence $\{\tm u^m\tm \}$ is bounded.\\
Setting $w^m:=u^m - u^{m-1}$, by virtue of  \eqref{definizione_schema_iterativo}, we arrive at
\begin{equation*}
w^{m+1}(t,x)=\nabla_x E \ast (w^m\otimes u^m)(t,x)-\nabla _x E \ast (u^{m-1}\otimes w^m)(t,x)
\end{equation*}
From Lemmas \ref{PPNLF}-\ref{stima_termine_convettivo_q}, and recalling estimate \eqref{A(ro,t)}, we arrive at the following sequence:
\begin{equation}\label{assoluta_convergenza}
\widetilde\tm w^1(t)\tm \leq c A^2(\rho,t),\dots,\widetilde\tm w^m(t)\tm \leq 2^{m-1}c^m  A^{m+1}(\rho,t)\dots
\end{equation}
By definition of $T(u_0)$, we can deduce that $A(\rho,t)< 1/(2c)<1$ for all $t \in (0,T(u_0))$.\par
Now, we consider the sequence
\be\label{series}u^{m}-u^0:=\mbox{$\underset{i=1}{\overset m\sum}$}w^i \quad \text{for}\,\, m\in \N\,.\ee
Via  estimates \rf{assoluta_convergenza}, the series on the right-hand side of \rf{series}  is absolutely convergent with respect to the functional $\tm \cdot\tm $. So, we get the convergence of the sequence $\{u^m\}$ with respect to the functional $\tm \cdot\tm $. We denote by $u$ the limit that in particular enjoys of estimates \rf{stima_sol_A(ro,t)} and \rf{andamento_in_q}. Let us explicitly point out that, by virtue  of estimates \rf{stima_sol_A(ro,t)} realted to $L^q$-norm and $L^\infty$-norm, for $u(t, x)$ the convolution product $\n_xE*(u\otimes u)$ is well posed. Then,   taking the difference $u(t, x) - u^
m(t, x)$, by the convergence and
estimates \rf{andamento_in_q} and \rf{A(ro,t)} (that is uniform in $m$) applying dominated convergence theorem,  in the limit on $m$ the following integral representation easily holds:
\begin{equation}\label{rappresentazione_integrale_soluzione}
u(t,x)=H\ast u_0(t,x)-\nabla_x E\ast (u\otimes u)(t,x)
\end{equation}
Furthermore, the uniform convergence of such sequence, that are continous function on $(\eta,T)\times \R^{n}$, ensures that the limit $u$ is a continous function in $(t,x)\in (\eta,T)\times \R^{n}$. \par
Now, let's reconsider estimate \eqref{A(ro,t)}. As in \cite{CM-L3}, we note that the following property holds true:\par
\begin{itemize}
\item[(P)] for any sequence $\{t_p\} \to 0$, one can construct a sequence $\{\rho_p\}\to 0$ such that $
1-4c_2\big(K(t,\rho)+  c\dm u_0\dm _{\alpha,p}\big)>0$ holds, and along these sequences, we get \[\lim_{p\to \infty} A(\rho_p,t_p)=0\]  
\end{itemize}
Therefore, we can again apply the Lemma \ref{lemma_iterate} to obtain
\begin{equation}\label{stima(P)}
\widetilde\tm u^m(t_p)\tm \leq A(\rho_p,t_p)
\end{equation}
Applying this property, recalling the definition of the functional $\tm \cdot\tm $, by virtue of estimate \eqref{stima(P)}, we deduce
the limit property \eqref{proprieta_limite}. \end{proof}
\begin{lemma}\label{andamento_a_0_parte_non_lineare}{\sl Let $u$ a vector field enjoying \rf{andamento_sol_q} and $\ov u$ defined by formula \rf{RIT}.  \begin{itemize}\item
If $p \in (n,2n)$, then, for all $q\in (p,2n)$, we get 
\begin{equation}\label{prop_L_parte_non_lineare}
\dm \overline{u}(t)\dm _{\frac{q}{2}}\leq c(u_0)\, t^{\frac{n}{q}-\frac{1}{2}}\qquad \mbox{for all }t\in(0,T)\,;
\end{equation}
\item If $p=2n$, then we get
\begin{equation}\label{proprieta_L_weak_parta_non_lineare1}
\dm u(t)\dm_{(n,\infty)}<\infty\,,\;t\in(0,T), \mbox{ and }\lim_{t\to 0}\dm\overline{u}(t)\dm _{(n,\infty)}=0\,;
\end{equation}
\item Finally, if $p>2n$, then we get
\begin{equation}\label{proprieta_L_weak_parta_non_lineare}
\dm \overline{u}(t)\dm _{(\frac{np}{2(p-n)},\infty)}\leq {c(u_0)}\,t^{\frac{1}{2}-\frac{n}{p}}\qquad \mbox{for all }t\in(0,T).
\end{equation}
\end{itemize} }
\end{lemma}
\begin{proof}
Let $p \in (n,2n)$ and $q\in (p,2n)$. In virtue of \eqref{andamento_sol_q} it holds:
\[\dm u(t)\dm _{q}\leq c\, t^{-\frac{1}{2}+\frac{n}{2 q}}\dm u_0\dm _{\alpha,p}\]
Employing the Minkowski inequality and  Young's inequality we get:
\[\dm \overline{u}(t)\dm _{\frac{q}{2}}\leq c\, \int_0^t \dm \nabla E(t-\tau)\dm _1\dm u\dm _{q}^2\,d\tau\leq c\, \dm u_0\dm _{\alpha,p}^2\int_0^t t^{-1+\frac{n}{q}}(t-\tau)^{-\frac{1}{2}}\leq c(u_0)\,t^{-\frac{1}{2}+\frac{n}{q}}.\]
\par
For $p=2n$,
using pointwise estimate \eqref{andamento_sol_q}, we get
\[|\overline{u}(t,x)|\leq  \frac{c}{|x|}\dm u_0\dm _{\alpha,p}^2(1+h\dm u_0\dm _{\alpha,p})^2\int_0^{t}{\tau^{-\frac{1}{2}}(t-\tau)^{-\frac{1}{2}}}\,d\tau\leq {c(u_0)}{|x|^{-1}}\,,\quad \mbox{for all }t\in(0,T)\,,\]
that implies
$$ \dm u(t) \dm_{(n,\infty)}\leq c(u_0) \,,\quad \mbox{for all }t\in(0,T)\,.$$
Also, for  $\varepsilon>0$, by virtue of limit property \eqref{limite_linfinito_soluzione}, for all $\{t_m\}$ such that $t_m \to 0$ as $m \to \infty$, there exists $\ov m$ so that
$$|u(t_m,x)|\leq \frac{\varepsilon}{t_m^{\frac14}|x|^{\frac12}}\dm u_0\dm _{\alpha,p}^2(1+h\dm u_0\dm _{\alpha,p})^2\qquad \mbox{for all }m>\ov m\,.$$
Using pointwise estimate \eqref{andamento_sol_q}, we get
\[|\overline{u}(t_m,x)|\leq  \frac{c\,\varepsilon}{|x|}\dm u_0\dm _{\alpha,p}^2(1+h\dm u_0\dm _{\alpha,p})^2\int_0^{t_m}{\tau^{-\frac{1}{2}}(t_m-\tau)^{-\frac{1}{2}}}\,d\tau\leq {c(u_0)}{|x|^{-1}}\varepsilon\,,\quad \mbox{for all }m>\ov m\,,\]
that implies
$$ \dm u(t_m) \dm_{(n,\infty)}\leq c(u_0)\varepsilon \,,\quad \mbox{for all }m>\ov m\,.$$
Letting $m\to\infty$, we deduce the property \eqref{proprieta_L_weak_parta_non_lineare1}.\par
Finally, if $p>2n$, employing again the pointwise estimate,  we arrive at
\[|\overline{u}(t,x)|\leq c\,\sfrac{\dm u_0\dm _{\alpha,p}^2(1+h\dm u_0\dm _{\alpha,p})^2}{|x|^{2-\frac{2n}{p}}}\int_0^t\!\!{\tau^{-\frac{n}{p}}(t-\tau)^{-\frac{1}{2}}}\,d\tau\leq {c(u_0)}{|x|^{-2+\frac{2n}{p}}}\,t^{\frac{1}{2}-\frac{n}{p}}\]
that leads to estimate \eqref{proprieta_L_weak_parta_non_lineare}.
\end{proof}

\begin{proof}[Proof of Theorem \ref{esistenza}]
In the hypotheses of Theorem \ref{esistenza}, by virtue of Lemma \ref{Lemma_8} and Lemma \ref{Lemma_10}, we establish a divergence free solution $u(t,x)$ solution to the integral equation \eqref{equazione_integrale} such that \eqref{andamento_sol_q} and \eqref{limite_linfinito_soluzione} hold. We can associate to $u(t,x)$ a pressure field defined by
\begin{equation}\label{def_pressione}
\pi(t,x):=-\nabla_{x_i} \int_{\R^n}\nabla_{x_j}\mathcal{E}(x-y)u_i(t,y)u_j (t,y)\,dy
\end{equation}
By Calderon-Zygmund theorem on singular integrals, we have:
\begin{equation}
\dm \pi(t)\dm _{\frac{q}{2}}\leq c\dm u(t)\dm _{{q}}^2 \quad \text{for all}\,\, t>0\,,
\end{equation}
from which we obtain the property \eqref{andamento_pressione}.
The limit property \eqref{weak_convergence1} is a consequence of the limit property \rf{convergenza_parte_lineare} for $u^0$  and Lemma\,\ref{andamento_a_0_parte_non_lineare} for $\ov u$. 
Finally, if we require $\dm u_0\dm _{\alpha,p}$ sufficiently small, since $\dm u_0\dm _{L_\alpha^p(B_{\rho})}\leq\dm u_0\dm _{\alpha,p}$ for all $\rho>0$, we can satisfy \eqref{istante_sup_esistenza} for arbitrary $\rho$ and then arbitrary $t$. This gives the stated global existence.
\end{proof}
Our uniqueness result is proved in norm $\dm \cdot\dm _{\frac{\ov q}2}$, where $\frac{\ov q}2$ is less than  $p$. This is an atypical approach to the uniqueness, that usually is proved in norms related to the  spaces of existence of the solutions. This approach to the uniqueness is possible because we relax our uniqueness just to nonlinear part, that is   the term \rf{RIT}. It is quite naturale to inquire why we need  this strategy. The strategy is employed because we have no condition of smallness on the data and contemporarily we need a norm for  $\ov u$ that goes to zero for $t\to0$. The {\it component} $\ov u$ is connected to the Oseen representation. Hence, it contains the nonlinear term. Evaluating $\ov u$ in $L^r$-norm, for some exponent $r>p$, due to  the nonlinear term we arrive at   estimate of the  square $L^{2r}$-norm  of the solution $u$. As a consequence the time integral is not  zero for $t\to0$. In order to around this difficulty we look for estimates in norm with exponent smaller than $p$. Hence considering norm of spaces that are off to the existence class. For our aims we need the following
\begin{lemma}\label{Dv}
Let $v(t,x)$ a solution to \eqref{problem} corrisponding to a initial datum $v_0$ and satisfying \eqref{weak_convergence1} and \eqref{andamento_sol_q}. Then, $v$ admits the decomposition:
\begin{equation}
v=v^0+\ov v\,,
\end{equation}
where $v^0=H\ast v_0$ and $\ov v=\n E \ast v\otimes v$. 
\end{lemma}
\begin{proof}
Setting $v^0:= H\ast v_0$, we define $w=v-v^0$. Furthermore, we consider $\ov v$ a solution to the following problem:
\begin{equation}\label{Svs}
\begin{cases}
\ov v_t -\Delta \ov v + \n \pi_{\ov v}=-v\cdot \n v \qquad &\text{in} \quad (0,\infty)\times \R^n\,, \\
\nabla \cdot \ov v(t,x)=0 \qquad &\text{in} \quad (0,\infty)\times \R^n \,,\\
\ov v(0,x)=0 \qquad &\text{on}\quad  \{0\}\times \R^n\,.
\end{cases}
\end{equation}
We know that $\ov v $ has the representation
$$\ov v(t,x) = \int_0^t\n E (t-\tau,x-y)v \otimes v(\tau,y)\,d\tau$$
As is easily  to verify, $w$ is an other solution to \eqref{Svs}. We claim that $w\equiv \ov v$. Indeed, setting $\ov w = w- \ov v$, $\ov w$ satisfies
$$\ov w_t -\Delta \ov w +\n \pi_{\ov w}=0\, ,$$
where $\pi_{\ov w}= \pi_{v}-\pi_{\ov v}$. \par Let $\psi$ be the solution to problem \rf{eq_calore} with initial datum $\psi_0\in \mathscr C_0(\R^n)$ and $a=b\equiv 0$. For $t>0$, we set $\widehat{\psi}(\tau,x):=\psi(t-\tau,x)$, provided that $(\tau,x)\in (0,t)\times \R^n$. The function $\widehat\psi$ is a  solution backward in time on $(0,t)\times\R^3$ with $\widehat{\psi}(t,x)=\psi_0(x)$, and it enjoys properties \rf{RHE}-\rf{SP}. Then, for all $t>s>0$, $\ov w$ satisfies the following integral equation:
\be\label{CdR}
(\ov w(t), \psi_0)=(\ov w(s), \psi(t-s))\, .
\ee
Now we distinguish three cases:
\begin{itemize}
\item[1)] For $p\in(n,2n)$, taking the absolute value and applying H\"older inequality, we have:
\begin{align*}
|(\ov w(t), \psi_0)|&= |(\ov w(s), \psi(t-s))|=|(\ov w(s), \psi(t-s)-\psi(t))|+|(\ov w(s), \psi(t))| \\
& \leq\dm  w(s) \dm_q \dm \psi(t-s)-\psi(t)\dm_{q'}+ \dm \ov v(s) \dm_{\frac{q}{2}} \dm \psi(t-s)-\psi(t)\dm_{\frac{q}{q-2}}+|(\ov w(s), \psi(t))|\\
&\leq s^{\frac{1}{2}+\frac{n}{2q}}\, c(v_0)c(\psi_0)+ c(\psi_0) \dm \ov v(s)\dm_{\frac{q}{2}} +|(\ov w(s), \psi(t))|\,,
\end{align*}
Since $\ov w= w-\ov v$, by virtue of  the  limit property \eqref{weak_convergence1} for $w$ and Lemma\,\ref{andamento_a_0_parte_non_lineare} for $\ov v$, letting $s\to 0$, we deduce that:
$$ |(\ov w(t), \psi_0)|=0,\quad \mbox{for all }\psi_0 \in \mathscr{C}_0(\R^n)\,,$$
then, $\ov w=0$ on $\{t\}\times \R^n$ for all $t>0$. So, by definition of $w$ we conclude that $v=v^0+\ov v$\,.\\
\item[2)] We consider $p=2n$. Again, taking the absolute value in \eqref{CdR}, applying H\"older inequality and taking in account \eqref{SPLS}, we have
\begin{align*}
|(\ov w(t), &\psi_0)|= |(\ov w(s), \psi(t-s))|=|(\ov w(s), \psi(t-s)-\psi(t))|+|(\ov w(s), \psi(t))| \\
& \leq\dm  w(s) \dm_q \dm \psi(t-s)-\psi(t)\dm_{q'}+ \dm \ov v(s) \dm_{(n,\infty)} \dm \psi(t-s)-\psi(t)\dm_{(n',1)}+|(\ov w(s), \psi(t))|\\
&\leq s^{\frac{1}{2}+\frac{n}{2q}}\, c(v_0)c(\psi_0)+ \dm\psi_0\dm_{(n',1)}\dm \ov v(s)\dm_{(n,\infty)} +|(\ov w(s), \psi(t))|\,.
\end{align*}
Since $\ov w= w-\ov v$, by virtue of the  limit property \eqref{weak_convergence1} for $w$ and Lemma\,\ref{andamento_a_0_parte_non_lineare} for $\ov v$,
letting $s\to 0$, we deduce that:
$$ |(\ov w(t), \psi_0)|=0,\quad \mbox{for all }\psi_0 \in \mathscr{C}_0(\R^n).$$
\end{itemize}
As in the  previous case, we deduce that $v=v^0+\ov v$.
\item[3)] Finally, for $p>2n$, from \eqref{CdR}, by same argument lines, we have
\begin{align*}
&|(\ov w(t), \psi_0)|= |(\ov w(s), \psi(t-s))|=|(\ov w(s), \psi(t-s)-\psi(t))|+|(\ov w(s), \psi(t))| \\
& \leq\dm  w(s) \dm_q \dm \psi(t-s)-\psi(t)\dm_{q'}+ \dm \ov v(s) \dm_{(\frac{np}{2(p-n)},\infty)} \dm \psi(t-s)-\psi(t)\dm_{(\frac{np}{p(n-2)+2n},1)}+|(\ov w(s), \psi(t))|\\
&\leq s^{\frac{1}{2}+\frac{n}{2q}}\, c(v_0)c(\psi_0)+ c\, \dm \ov v(s)\dm_{(\frac{np}{2(p-n)},\infty)}\dm \psi_0\dm_{(\frac{np}{p(n-2)+2n},1)} +|(\ov w(s), \psi(t))|\,.
\end{align*}
Since $\ov w= w-\ov v$, by virtue of the limit property \eqref{weak_convergence1} for $w$ and Lemma\,\ref{andamento_a_0_parte_non_lineare} for $\ov v$,
letting $s\to 0$, we deduce the thesis also in this case. 
\end{proof} 
Now, we are  in a position to prove Theorem\,\ref{unicita} 
\begin{proof}[Proof of Theorem \ref{unicita}] Our result of uniqueness compare two solutions $u(t, x)$ and $v(t, x)$ where $u$ is a solution of Theorem\,\ref{esistenza} and $v$ enjoys the same regularity. By virtue of Lemma\,\ref{Dv}, solution $v$ admits a decomposition of the kind $v=u^0+\ov v$ where, depending on $p$, we  require to function  $\ov v$ the estimates \rf{prop_L_parte_non_lineare}, \rf{proprieta_L_weak_parta_non_lineare1} and \rf{proprieta_L_weak_parta_non_lineare}. \par 
We employ a duality argument as suggested in \cite{F} with a slight modification. 
Set $w(t,x):=v(t,x)- u(t,x)=\overline{v}(t, x)-\overline{u}(t, x)$ and $\pi_w:=\pi_v-\pi_u$. Recalling the regularity of $w$ for $t>0$, for all $t>s>0$, the pair $(w,\pi_w)$ solves the integral equation:
\begin{equation}\label{WFFW}\ba{ll}
(w(t),\phi(t))\hskip-0.2cm&\displ= (w(s),\phi(s))+ \int_s^t \big[(w,\phi_\tau+\Delta\phi)\VSE\hskip3.2cm+(w\cdot \nabla \phi  ,\overline{v})+(\overline{u} \cdot \nabla \phi,w)+(u^0\cdot \nabla \phi,w)+w\cdot\nabla\phi,u^0\big]d\tau\,,\ea
\end{equation} for all $\phi\in C([0,T);J^{q'}(\R^n))\cap L^{q'}(0,T;W^{2,q'}(\R^n))$ with $\phi_t\in L^{q'}(0,T;L^{q'}(\R^n))$, and $t>s>0$\,. Let $\psi$ be the solution to problem \rf{eq_calore} with initial datum $\psi_0\in \mathscr C_0(\R^n)$ and $a=b\equiv 0$. For $t>0$, we set $\widehat{\psi}(\tau,x):=\psi(t-\tau,x)$, provided that $(\tau,x)\in (0,t)\times \R^n$. The function $\widehat\psi$ is a  solution backward in time on $(0,t)\times\R^3$ with $\widehat{\psi}(t,x)=\psi_0(x)$, and it enjoys properties \rf{RHE}-\rf{SP}. Then, for all $t>s>0$, $w$ satisfies the following integral equation:
\begin{equation}\label{ultima}
(w(t),\psi_0)= (w(s),\widehat{\psi}(s))+ \int_s^t [(w\cdot \nabla \widehat{\psi}  ,\overline{v})+(w\cdot\n \widehat \psi, u^0)+(\overline{u} \cdot \nabla \widehat{\psi},w)+(u^0\cdot \nabla \widehat{\psi},w) ]\,d\tau.
\end{equation}
Now, in order to employ Lemma\,\ref{andamento_a_0_parte_non_lineare} we distinguish three cases: 
\begin{itemize}
\item[1)]Let $u_0 \in L^p_{\alpha}(\R^n)$ with $p\in (n,2n)$.  In this case, we can consider an exponent $\overline{q}\in (p,2n)$ in order to   employ \rf{prop_L_parte_non_lineare} with $\frac{\ov q}2$ for both the solutions. Hence, in particular we get $\displ\lim_{s\to0}\dm w(s)\dm_{\frac{\ov q}2}=0$. From \eqref{ultima} and from the estimate related to $\dm \nabla \widehat{\psi}\dm _{\frac{\overline{q}}{\overline{q} - 2}}$  we get
\be\label{ultima2}\ba{ll}
|(w(t),\psi_0)|\hskip-0.2cm&\displ\leq |(w(s),{\psi}(t-s)| + c\sup_{(s,t)}[\tau^{\frac{1}{2}}(\dm \overline{u}\dm _{\infty}+ \dm \overline{v}\dm _{\infty}+\dm u^0\dm _{\infty})] \notag\\
&\displ\hskip2cm\times \sup_{(s,t)}\dm w(\tau)\dm _{\frac{\overline{q}}{2}}\int_s^t \tau^{-\frac{1}{2}}\dm \nabla \widehat{\psi}\dm _{\frac{\overline{q}}{\overline{q}-2}}\,d\tau \notag\\
&\displ\leq \dm w(s)\dm _{\frac{\overline{q}}{2}}\dm {\psi_0}\dm _{\frac{\overline{q}}{\overline{q}-2}} + c\sup_{(s,t)}[\tau^{\frac{1}{2}}(\dm \overline{u}\dm _{\infty}+ \dm \overline{v}\dm _{\infty}+\dm u^0\dm _{\infty})] \notag\\
&\displ\hskip2cm\times\dm \psi_0\dm _{\frac{\overline{q}}{\overline{q}-2}}\sup_{(s,t)}\dm w(\tau)\dm _{\frac{\overline{q}}{2}}\int_s^t \tau^{-\frac{1}{2}}(t-\tau)^{-\frac{1}{2}}\,d\tau\,,
\ea\ee
for all $t \in [0,T(u)) \cap [0,T(v))$. Since $\psi_0$ is arbitrary, letting $s \to 0$, we obtain
\[ \dm w(t)\dm _{\frac{\overline{q}}{2}}\leq  c\sup_{(0,t)}[\tau^{\frac{1}{2}}(\dm \overline{u}\dm _{\infty}+ \dm \overline{v}\dm _{\infty}+\dm u^0\dm _{\infty})]\sup_{(0,t)}\dm w(\tau)\dm _{\frac{\overline{q}}{2}}\]
From the validity of the first limit property in \eqref{limite_linfinito_soluzione}, one easily deduces the uniqueness on some interval $(0,\delta]$. It remains to discuss the uniqueness for $t\geq \delta$.\par
Since for $t>0$ we have $w$ regular,  we are going to consider  formula  \eqref{ultima} with $s=\delta$. Since $\psi$ enjoys \rf{RHE}, and $\dm w(\delta)\dm _{\frac{\overline{q}}{2}}=0$, for all $r\in[\frac{\ov q}2,\infty)$, we deduce 
\be\label{WB} \dm w(t)\dm _{r}\leq  c\, \delta^{-\frac{1}{2}}\sup_{(\delta,t)}[\tau^{\frac{1}{2}}(\dm \overline{u}\dm _{\infty}+ \dm \overline{v}\dm _{\infty}+\dm u^0\dm _{\infty})]\int_{\delta}^t \dm w(\tau
)\dm _{r}(t-\tau)^{-\frac{1}{2}}\,d\tau\,.\ee
Applying Lemma\,\ref{L-GWSI} with $h(\tau)=\dm w(\tau)\dm _{r}$, we obtain $\dm w(t)\dm _{r}=0$  for all   $t\in[\delta,T)$.\item[2)] Let $u_0 \in L^{2n}_{\alpha}(\R^n)$. By virtue of \rf{proprieta_L_weak_parta_non_lineare1}, we discuss the uniqueness in the quasi-norm $\dm \cdot\dm _{{(n,\,\infty)}}$. From \eqref{ultima} and from the estimate related to $\dm \nabla \widehat{\psi}\dm _{(n',1)}$ given in \eqref{SPLS},  we get:
\be\label{ultima3}\ba{ll}
|(w(t),\psi_0)|\hskip-0.2cm&\displ\leq |(w(s),{\psi}(t-s)| + c\sup_{(s,t)}[\tau^{\frac{1}{2}}(\dm \overline{u}\dm _{\infty}+ \dm \overline{v}\dm _{\infty}+\dm u^0\dm _{\infty})] \notag\\
&\displ\hskip2cm\times \sup_{(s,t)}\dm w(\tau)\dm _{\frac{\overline{q}}{2}}\int_s^t \tau^{-\frac{1}{2}}\dm \nabla \widehat{\psi}\dm _{(n',1)}\,d\tau \notag\\
&\displ\leq \dm w(s)\dm _{\frac{\overline{q}}{2}}\dm {\psi}_0\dm _{(n',1)} + c\sup_{(s,t)}[\tau^{\frac{1}{2}}(\dm \overline{u}\dm _{\infty}+ \dm \overline{v}\dm _{\infty}+\dm u^0\dm _{\infty})] \notag\\
&\displ\hskip2cm\times\dm \psi_0\dm _{(n',1)}\sup_{(s,t)}\dm w(\tau)\dm _{\frac{\overline{q}}{2}}\int_s^t \tau^{-\frac{1}{2}}(t-\tau)^{-\frac{1}{2}}\,d\tau\,,
\ea\ee
for all $t \in [0,T(u)) \cap [0,T(v))$. Since $\psi_0$ is arbitrary, we obtain:
\begin{align*}
\dm w(t)\dm _{(n,\infty)}\leq c\sup_{(0,t)}[\tau^{\frac{1}{2}}(\dm \overline{u}\dm _{\infty}+ \dm \overline{v}\dm _{\infty}+\dm u^0\dm _{\infty})]\sup_{(0,t)}\dm w(\tau)\dm _{(n,\infty)}.
\end{align*}
The validity of the limit property in \eqref{limite_linfinito_soluzione} ensures the uniqueness on some interval $(0,\delta]$.  Moreover,   for $t>\delta$   we can argue as already made for  \rf{WB}. In this case we take into account that 
for  $s=\delta$, without loss the generality,  $\dm  w(\delta)\dm _{r} =0$, and then, for $r\in(p,\infty)$, we set 
\be \dm  w(t)\dm _{r} \leq  c\, \delta^{-\frac{1}{2}}\sup_{(\delta,t)}[\tau^{\frac{1}{2}}(\dm \overline{u}\dm _{\infty}+ \dm \overline{v}\dm _{\infty}+\dm u^0\dm _{\infty})]\int_{\delta}^t \dm  w(\tau)\dm _{r} (t-\tau)^{-\frac{1}{2}}\,d\tau.\ee
Applying Lemma\,\ref{L-GWSI} with $h(\tau)=\dm w(\tau)\dm _{r}$, we obtain $\dm  w(t)\dm _{r}=0$\, for all  $t\in[\delta,T)$.
\item[3)] $u_0 \in L^p_{\alpha}(\R^n)$ with $p>2n$. We consider the norm $\dm \cdot\dm _{(\frac{np}{2(p-n)},\infty)}$.
From \eqref{ultima},  we get:
\begin{align*}
|(w(t),\psi_0)|&\leq |(w(s),{\psi}(t-s)| + c\sup_{(s,t)}[\tau^{\frac{1}{2}}(\dm \overline{u}\dm _{\infty}+ \dm \overline{v}\dm _{\infty}+\dm u^0\dm _{\infty})]\times \notag\\
&\hskip2cm\times \sup_{(s,t)}\dm w(\tau)\dm _{(\frac{np}{2(p-n)},\infty)}\int_s^t \tau^{-\frac{1}{2}}\dm \nabla \widehat{\psi}\dm _{(\frac{np}{np-2(p-n)},1)}d\tau \notag\\
&\leq \dm w(s)\dm _{(\frac{np}{2(p-n)},\infty)}\dm \widehat{\psi}(s)\dm _{(\frac{np}{np-2(p-n)},1)} + c\sup_{(s,t)}[\tau^{\frac{1}{2}}(\dm \overline{u}\dm _{\infty}+ \dm \overline{v}\dm _{\infty}+\dm u^0\dm _{\infty})]\times \notag\\
&\hskip2cm\times\dm \psi_0\dm _{(\frac{np}{np-2(p-n)},1)}\sup_{(s,t)}\dm w(\tau)\dm _{(\frac{np}{2(p-n)},\infty)}\int_s^t \tau^{-\frac{1}{2}}(t-\tau)^{-\frac{1}{2}}\,d\tau\,.
\end{align*}
Taking   property \eqref{proprieta_L_weak_parta_non_lineare} into account and repeat the same argument lines as in the case 1)-2), we get $\dm w(t)\dm _{(\frac{np}{2(p-n)},\infty)}=0$ \,for all $t\in (0,T)$.
\end{itemize}
The result is proved.
\end{proof}
\section{Appendix} 
For $\beta \in \big(-\frac{n}{q},\frac{n}{q'}\big)$, we set:
\begin{gather*}
J^q_{\beta}(\mathbb{R}^n)=\{v\in L^q_{\beta}(\mathbb{R}^n)\, |\, (v,\nabla h)=0 \,\, \text{for all}\, h \in W^{1,q'}_{loc}(\mathbb{R}^n,\beta ')\text{,}\, \nabla h \in L^{q'}_{\beta'}(\R^n)\}\\
G^q_{\beta}(\mathbb{R}^n)=\{u\in L^q_{\beta}(\mathbb{R}^n)\, |\, \exists h\, : \,\, u=\nabla \pi \text{,}\,\text{with} \, \pi \in W^{1,q}_{loc}(\mathbb{R}^n,\beta ')\text{,}\, \nabla \pi \in L^{q}_{\beta}(\R^n)\} 
\end{gather*}
Hereafter, with $\beta'$ we denote the dual weigth , that is $\beta'=-\beta$.\par
Let us remark that in the discussion of the following results, we take $\beta \neq 0$. Indeed, the case $\beta=0$ corrisponds to the one described by the classical theory of Helmholtz decomposition. 
\begin{lemma} \label{lemma}{\sl
Let $q \in (1,+\infty)$ and $ \beta \in \big(-\frac{n}{q},\frac{n}{q'}\big)$. Then we get
\begin{equation}
J^q_{\beta}(\mathbb{R}^n)\cap G^q_{\beta}(\mathbb{R}^n)=\{0\}
\end{equation}}
\end{lemma}
\begin{proof}
Assume that exists $w \in J^q_{\beta}(\mathbb{R}^n)\cap G^q_{\beta}(\mathbb{R}^n) $ such that $w\neq 0$. Then, by definition:
\begin{gather*}
(w,\nabla h)=0 \quad \text{for all}\, h \in W^{1,q'}_{loc}(\mathbb{R}^n,\beta ')\text{,}\, \nabla h \in L^{q'}_{\beta'}(\R^n) \\
w=\nabla \pi \quad \text{with} \, \pi \in W^{1,q}_{loc}(\mathbb{R}^n,\beta )\text{,}\, \nabla \pi \in L^{q}_{\beta}(\R^n)
\end{gather*}
Let $h_R \in (0,1)$ be a smooth cutoff function with:
\begin{equation*}
h_R =
\begin{cases}
1 \quad \text{if}\,\, |x|\leq R \\
0 \quad \text{if}\,\, |x|\geq 2R
\end{cases}
\end{equation*}
and $\nabla h_R=O(\frac{1}{R})$.\par
We denote by $\psi$ the solution of the problem:
\begin{equation} \label{problem1}
\Delta u= \nabla \cdot g
\end{equation}
with $g\in C^{\infty}_0(\R^n)$.
We know that
\begin{equation}
\psi = \nabla \cdot \int_{\R^n} \mathcal{E}(x-y)g(y)\,dy
\end{equation}
and
\begin{equation}\label{86}
\nabla \psi= \int^{*}_{\R^n} \nabla \nabla \mathcal{E}(x-y)g(y)\,dy
\end{equation}
where $\mathcal{E}$ is the fundamental solution of the Laplace equation.\par
We can employ Calder\'on-Zygmund theorem to obtain the following estimate:
\begin{equation}
\dm \nabla \psi\dm _{q'}\leq C\dm g\dm _{q'} \quad \text{for all}\, q' \in (1,+\infty).
\end{equation}
Moreover, by hypothesis on $\beta$, we can apply the result due to E. M. Stein in \cite{Stein}
\begin{equation}\label{stein}
\dm \nabla \psi |x|^{\beta}\dm _{q'}\leq A(q',\beta) \dm g |x|^{\beta}\dm _{q'}
\end{equation}
Considering  $R> diam(supp\,g)$, the following computation holds:

\begin{align*}
0&=(w,\nabla\psi)=(\nabla \pi,\nabla \psi)=(h_R \nabla \pi, \nabla \psi) + ( (1-h_R)\nabla \pi,\nabla \psi)\\
 &= -((\pi-\overline{\pi}_R)h_R,\nabla \cdot g)+ ((\pi-\overline{\pi}_R)\nabla h_R,\nabla \psi)+( (1-h_R)\nabla \pi,\nabla \psi)\\
&= (\nabla\pi,g)-((\pi-\overline{\pi}_R)\nabla h_R,\nabla \psi)+  ((1-h_R)\nabla \pi,\nabla \psi)= I_1 + I_2 + I_3
\end{align*}
where $\overline{\pi}_R=\int_{B_{2R}(O)}\pi(y) \, dy$.\par
Trivially, $I_3$ tends to zero for $R\to +\infty$.\par
Now, we estimate $I_2$. We set $K_R=B_{2R}(O)\setminus B_{R}(O)$.\par
For $\beta \in (-\frac{n}{q},0)$, we get:
\begin{align*}
|I_2|&\leq \sfrac{1}{R}\dm (\pi-\overline{\pi}_R)\dm _{L^q(B_{2R})} \dm \nabla \psi\dm _{{L^{q'}(K_R)}}\leq \sfrac{C}{R}R \dm \nabla\pi\dm _{L^q(B_{2R})} \dm \nabla \psi\dm _{{L^{q'}(K_R)}}\\
&\leq  C\big\dm  \sfrac{\nabla\pi}{|x|^{\beta}}|x|^{\beta}\dm  _{L^q(B_{2R})}\dm  \sfrac{\nabla \psi}{|x|^{\beta'}}|x|^{\beta'}\dm  _{{L^{q'}(K_R)}}\leq C R^{\beta'}\dm \nabla \pi |x|^{\beta}\dm _{L^q(B_{2R})} \sfrac{1}{R^{\beta'}}\dm \nabla \psi|x|^{\beta'}\dm _{{L^{q'}(K_R)}}.
\end{align*}
This last estimate ensure that $I_2$ tends to zero for $R\to +\infty$.\par
Now, let $\beta \in (0,\frac{n}{q'}$. Being $g$ with compact support, from \eqref{86} we get $|\nabla \psi |\leq \frac{B}{|x|^n}$ for all $x\in \R^n$ such that $|x|>2\,diam(supp\,g)$. So, applying H\"older's inequality,  we get:
\begin{align*}
|I_2|&\leq \sfrac{1}{R}\dm \nabla \psi\dm _{L^{\infty}(K_R)}\int_{B_{2R}(O)}|\pi - \overline{\pi}_R|\,dx\leq \sfrac{C}{R}R\dm \nabla \psi\dm _{L^{\infty}(K_R)} \int_{B_{2R}(O)}|\nabla \pi|\,dx  \\
&= C\dm \nabla \psi\dm _{L^{\infty}(K_R)}\int_{B_{2R}(O)}\sfrac{|\nabla \pi|}{|x|^{\beta}}|x|^{\beta}\,dx \leq C\dm \nabla \psi\dm _{L^{\infty}(K_R)}\dm \nabla \pi |x|^{\beta}\dm _q \Big[\int_{B_{2R}(O)}\sfrac{1}{|x|^{\beta q'}}\Big]^{\frac{1}{q'}}\\
&\leq C\dm \nabla \pi |x|^{\beta}\dm _q R^{-\frac{n}{q}-\beta}
\end{align*}
which tends to zero for $R\to +\infty$.\par
So, for all $\beta \in (-\frac{n}{q},0)\cup(0,\frac{n}{q'})$ we deduce that:
\begin{equation}
(\nabla \pi,g)=0,
\end{equation}
for all $g\in C^{\infty}_0(\R^n)$. Thus $w=0$ and this contraddiction implies the thesis.
\end{proof}

We are in position to prove Helmhotz decomposition for $L^q_{\beta}(\R^n)$-spaces:
\begin{theorem}\label{dec_helmotz_teorema}{\sl 
Let $q \in (1,+\infty)$ and $\beta \in (-\frac{n}{q},\frac{n}{q'})$. Then we get 
\begin{equation}
L^q_{\beta}(\R^n)=J^q_{\beta}(\mathbb{R}^n)\oplus G^q_{\beta}(\mathbb{R}^n)\,,
\end{equation}
 that is for all $u \in L^q_{\beta}(\R^n)$, $u=v+\nabla \pi_u$, with the following integral identities:
\begin{gather}
(u,\nabla \pi)=(\nabla \pi_u ,\nabla \pi) \quad \text{for all}\, \, \nabla \pi \in G_{\beta '}^{q'}(\R^n) \label{1)}\\
(v,\nabla \pi)=0 \quad \text{for all}\, \, \nabla \pi \in G_{\beta '}^{q'}(\R^n) \label{2)}
\end{gather}
and
\begin{equation} \label{3}
\dm v\dm _{q,\beta} + \dm  \nabla \pi \dm _{q,\beta} \leq C \dm u\dm _{q,\beta}
\end{equation}
whith $C$ indipendent of $u$.}
\end{theorem}
\begin{proof}
Let us first considering a field $u \in C^{\infty}_0(\R^n)$. As we have done before, we denote by $\psi$ the solution of the problem $\Delta \psi=\nabla \cdot u$ and for such solution \eqref{stein} holds. We set:
\begin{equation}\label{setting}
v:= u- \nabla \psi
\end{equation}
We claim that $(v,\nabla \pi)=0$ for all $\pi \in W^{1,q'}_{loc}(\R^n,\beta')$, with $\nabla \pi \in L^{q'}_{\beta'}(\R^n)$. Considering $h_R$ and $\overline{\pi}_R$ as in lemma \ref{lemma} with $R>diam(supp\,u)$, then:
\begin{align*}
(v,\nabla \pi)&=(v,h_R \nabla (\pi - \overline{\pi}_R) + (v,(1-h_R)\nabla \pi)=\\
&=-(\nabla \cdot v \, h_R,(\pi-\overline{\pi}_R)) - (v\cdot \nabla h_R,\pi-\overline{\pi}_R)+(v,(1-h_R)\nabla \pi)
\end{align*}
As we saw, the first term on the right hand is zero. Considering $|(v, \nabla \pi)|$, is enough to argue as in the case of Lemma \ref{lemma} to find that $(v, \nabla \pi)=0$. So, $v \in J^q_{\beta}(\mathbb{R}^n)$ and this fact implies \eqref{2)}. From \eqref{2)} and \eqref{setting}, we obtain \eqref{1)} and \eqref{3}.\par
Now, we consider $u \in L^{q}_{\beta}$. There exists a sequence $\{u_k\}\subset C^{\infty}_0(\R^n)$ converging to $u$ in $L^{q}_{\beta}$-norm. Thus:
\begin{equation*}
u_k=v_k+\nabla \psi_k
\end{equation*}
with $v_k \in J^{q}_{\beta}(\R^n)$ and $\nabla \psi_k \in G^{q}_{\beta}(\R^n)$. By virtue of \eqref{3}, we get:
\begin{equation*}
\dm v_k-v_m\dm _{q,\beta}+\dm \nabla \psi_k-\nabla \psi_m\dm _{q,\beta}\leq C\dm u_k-u_m\dm _{q,\beta}
\end{equation*}
In particular, $v_k$ converging to a field $v$ and $\nabla \psi_k$ converging to a field $h$ in $L^{q}_{\beta}$-norm. We claim that $v \in J^{q}_{\beta}(\R^n)$ and $h \in G^{q}_{\beta}(\R^n)$. As the matter of fact, we get:
\begin{equation*}
(v, \nabla p)=\lim_{k}(v_k, \nabla p)=0
\end{equation*}
for all $p \in W^{1,q'}_{loc}(\R^n,\beta')$, with $\nabla p \in  L^{q'}_{\beta'}(\R^n)$, and considering $\varphi \in \mathscr{C}_0(\R^n)$, we have:
\begin{equation*}
(h,\varphi)=\lim_{k}(\nabla \psi_k, \varphi)=-\lim_{k}(\psi_k, \nabla\cdot \varphi)=0
\end{equation*}
Thus, there exists  $\pi_u \in W^{1,q}_{loc}(\R^n,\beta)$ such that $h=\nabla \pi_u$. Moreover, applying Minkowski's inequality, we arrive at
\begin{align*}
\dm u-(v+\nabla \pi_u)\dm _{q,\beta}\leq \dm u- u_k +u_k-(v+\nabla \pi_u)\dm _{q,\beta}\leq \dm u- u_k\dm _{q,\beta} + \dm v_k-v \dm _{q,\beta} +\dm  \nabla \psi_k -\nabla \pi_u\dm _{q,\beta}
\end{align*}
which tends to zero for $k\to +\infty$.\par
Then, we get
\begin{equation*}
u=v+\nabla \pi_u
\end{equation*}
and
\begin{equation*}
\dm v\dm _{q,\beta}+ \dm \nabla \pi_u\dm _{q,\beta}=\lim_{k}\Big(\dm v_k\dm _{q,\beta}+ \dm \nabla \psi_k\dm _{q,\beta}\Big)\leq C \lim_{k}\dm u_k\dm _{q,\beta}= C\dm u\dm _{q,\beta}
\end{equation*}
This fact complete the proof.
\end{proof}
\begin{lemma}{\sl 
Let $q\in (1,+\infty)$ and $\beta \in (-\frac{n}{q},\frac{n}{q'})$. Then, $J_{\beta}^q(\R^n)\equiv$ completion of $\mathscr{C}_0(\R^n)$ in $L_{\beta}^q$-norm holds.}
\end{lemma}
\bp See \cite{GPG}, \cite{SS} or \cite{HDS} \ep

\end{document}